\documentclass{amsart}

\usepackage[T1]{fontenc}
\usepackage[utf8]{inputenc}
\usepackage[english]{babel}
\usepackage{newpxtext} %
\usepackage{eulervm} %

\usepackage{amssymb,textcomp,mathrsfs,stmaryrd,braket,commath,relsize}
\usepackage{bm} %
	\SetSymbolFont{stmry}{bold}{U}{stmry}{m}{n} %
\usepackage[new]{old-arrows} %
\usepackage{xcolor} %

\usepackage{mathtools}
	\mathtoolsset{showonlyrefs} %

\theoremstyle{plain} %
	\newtheorem{theorem}{Theorem}[section]
	\newtheorem*{theorem*}{Theorem}
	\newtheorem{proposition}{Proposition}[section]
	\newtheorem*{proposition*}{Proposition}
	\newtheorem{corollary}{Corollary}[section]
	\newtheorem*{corollary*}{Corollary}
	\newtheorem{lemma}{Lemma}[section]
	\newtheorem*{lemma*}{Lemma}
	\newtheorem{conjecture}{Conjecture}[section]
	\newtheorem*{conjecture*}{Conjecture}

\theoremstyle{definition} %
	\newtheorem{definition}{Definition}[section]
	\newtheorem*{definition*}{Definition}

\theoremstyle{remark} %
	\newtheorem{remark}{Remark}[section]
	\newtheorem*{remark*}{Remark}
	\newtheorem{example}{Example}[section]
	\newtheorem*{example*}{Example}

\usepackage{etoolbox}
	\newcommand{\addQEDstyle}[2]{\AtBeginEnvironment{#1}{\pushQED{\qed}\renewcommand{\qedsymbol}{#2}}
	\AtEndEnvironment{#1}{\popQED}} %
	\addQEDstyle{remark}{$\triangle$}
	\addQEDstyle{remark*}{$\triangle$} %
	\addQEDstyle{example}{$\triangle$}
	\addQEDstyle{example*}{$\triangle$} %

\usepackage{tikz}
	\usetikzlibrary{cd} %
	\usetikzlibrary{braids} %
	\tikzstyle{vertex}=[circle,fill=black,minimum size=5pt,inner sep=0pt] %
	\newcommand{\vertex}{\node[vertex]}

\usepackage{dynkin-diagrams}

\usepackage{float} %

\AtBeginDocument{\def\MR#1{}} %
\apptocmd{\sloppy}{\hbadness 10000\relax}{}{} %

\makeatletter %
	\@namedef{subjclassname@2020}{
		\textup{2020} Mathematics Subject Classification}
\makeatother

\usepackage{todo}

\usepackage{mymacros}

\usepackage[psdextra,pdfencoding=auto,backref=page]{hyperref}
	\hypersetup{colorlinks,linkcolor={red!50!black},filecolor={green!50!black},urlcolor={blue!80!black},citecolor={blue!50!black}}

\begin{document}
\title[Local WMCG\MakeLowercase{s} and cabled braids]{Local wild mapping class groups and cabled braids} %

\author[J.~Douçot]{Jean Douçot} %
\address[J.~Douçot]{Department of Mathematics, University of Lisbon, Campo Grande, Edifício C6, PT-1749-016 Lisboa (Portugal)}
\curraddr{Institut de Recherche Mathématique Avancée, 7 rue René Descartes, 67084 Strasbourg Cedex, France
}
\email{jean.doucot@math.unistra.fr}
\thanks{J.~D. was funded by FCi\^encias.ID;
	and now by the PNRR Grant CF 44/14.11.2022,
	``Cohomological Hall Algebras of Smooth Surfaces and Applications'',
	led by O.~Schiffmann.}

\author[G.~Rembado]{Gabriele Rembado}
\address[G.~Rembado]{Hausdorff Centre for Mathematics, University of Bonn, 60 Endenicher Allee, D-53115 Bonn (Germany)}
\curraddr{Departament de Matemàtiques,
	Universitat Politècnica de Catalunya (UPC),
	Campus Diagonal Sud,
	Av. Doctor Marañón 44-50,
	08028 Barcelona,
	Spain}
\email{gabriele.rembado@upc.edu}
\thanks{G.~R. was supported by the Deutsche Forschungsgemeinschaft (DFG, German Research Foundation) under Germany’s Excellence Strategy - GZ 2047/1, Projekt-ID 390685813;
	and now by the European Union's Horizon Europe research and innovation programme under the Marie Skłodowska-Curie grant agreement No 101108575 (project~\href{https://cordis.europa.eu/project/id/101108575}{QuantMod}, ``Quantisation of moduli spaces: Hitchin connections and isomonodromic deformations'').}

\author[M.~Tamiozzo]{Matteo Tamiozzo}
\address[M.~Tamiozzo]{Department of Mathematics, University of Warwick, Zeeman Building, CV47AL Coventry (UK)}
\curraddr{Université Sorbonne Paris Nord, Institut Galilée, 99 Av. Jean Baptiste Clément, 93430 Villetaneuse (France)}
\email{tamiozzo@math.univ-paris13.fr}
\thanks{M.~T. was funded by ERC grant 804176}



\begin{abstract}
	We define and study some generalisations of pure $\mf g$-braid groups, for any complex reductive Lie algebra $\mf g$.
	They naturally occur in the theory of isomonodromic deformations for meromorphic connections with irregular singularities on principal bundles over Riemann surfaces, covering the general untwisted case, going beyond the case of generic irregular types.
	These generalised braid groups make up local pieces of the wild mapping class groups, which in turn extend the usual mapping class groups and govern the braiding of Stokes data.

	We establish a general product decomposition for these local wild mapping class groups, and in all classical cases define a fission tree governing the decomposition; in particular in type $D$ we will find a factor which is not isomorphic to any pure braid group coming from a root system.
	In type $A$, the fission tree and the pure braid group operad yields a proof of the corresponding ``multi-scale'' braiding conjecture.
\end{abstract}

{\let\newpage\relax\maketitle} %

\setcounter{tocdepth}{1}  %
\tableofcontents

\section{Introduction}

\renewcommand{\thetheorem}{\arabic{theorem}} %

\subsection{General aim}

In this article, we initiate a series of works with the aim of systematically studying the topology of the spaces of times for isomonodromic deformations of meromorphic connections with irregular singularities of any irregular type (encompassing in particular all the Painlevé equations).

This is motivated by the fact that the much-studied actions of mapping class groups on character varieties (also) arise from isomonodromic deformations of meromorphic connections with \textit{regular} singularities (i.e., simple poles):
and there is an analogous theory in the setting of meromorphic connections with \textit{irregular} singularities (i.e., higher-order poles),
which leads to wild mapping class group actions on wild character varieties~\cite{boalch_2014_geometry_and_braiding_of_stokes_data_fission_and_wild_character_varieties,boalch_yamakawa_twisted_wild_character_varieties}.
Our main aim is to describe as explicitly as possible,
and intrinsically,
the wild mapping class groups obtained in this way;
and later their actions.

As a first step, here we consider \textit{pure local} wild mapping class groups---defined in \S~\ref{sec:local_WMCG}---for \emph{untwisted} irregular types. They turn out to yield generalisations of pure $\mf g$-braid groups, and for classical Lie algebras they are controlled by a \emph{fission tree} which encodes (the admissible deformations of) the irregular type.

In this introduction:
(i) we explain the context and motivation of our work (cf.~\S~\ref{sec:background});
(ii) we spell out the structure of this paper and state its main results (cf.~\S~\ref{sec:main_results});
and (iii) we illustrate some of these results in a concrete example (cf.~\S~\ref{sec:ex-typeA}).

\subsection{Nonlinear monodromy actions in 2d gauge theory}
\label{sec:background}

Let $\Sigma$ be a compact Riemann surface, $\bm a \sse \Sigma$ a finite subset with open complement $\Sigma^{\circ} \ceqq \Sigma \setminus \bm a$,
and $G$ a complex reductive Lie group (for instance, $G = \GL_n(\mb C)$).
Much is known about the (tame) character variety
\begin{equation}
	\label{eq:tame_betti}
	\mc M_{\on B}(\Sigma,\bm a) \ceqq \Hom \bigl( \pi_1(\Sigma^{\circ}),G \bigr) \bs\!\!\bs G,
\end{equation}
parameterising---%
isomorphism classes of---%
$G$-local systems on $\Sigma^o$.
The Riemann--Hilbert correspondence~\cite{deligne_1970_equations_differentielles_a_points_singuliers_reguliers} (cf.~\cite{katz_1976_an_overview_of_deligne_s_work_on_hilbert_s_twenty_first_problem}), sending a meromorphic connection on a principal $G$-bundle over $\Sigma$ with regular singularities at $\bm a \sse \Sigma$ to its monodromy representation, is a transcendental map relating the \textit{de Rham} space $\mc M_{\dR}(\Sigma,\bm a)$ (parametrising such connections) with the \textit{Betti} space $\mc M_{\on B}(\Sigma,\bm a)$.
The character variety~\eqref{eq:tame_betti} is an algebraic Poisson variety, and the mapping class group of $(\Sigma,\bm a)$ acts on it by algebraic Poisson automorphisms. In fact, for any smooth family $\ul{\Sigma} \to \bm B$ of Riemann surfaces with marked points the character varieties of the fibres assemble into a fibration endowed with a flat (complete) Ehresmann connection (the nonabelian Gau\ss{}--Manin connection~\cite{simpson_1994_moduli_of_representations_of_the_fundamental_group_of_a_smooth_projective_variety_i,simpson_1994_moduli_of_representations_of_the_fundamental_group_of_a_smooth_projective_variety_ii}). Via the monodromy of this nonlinear connection, the fundamental group of the base $\bm B$ hence acts on the character variety~\eqref{eq:tame_betti} of any fibre, by algebraic Poisson automorphisms.\fn{
	Recall that Hitchin~\cite{hitchin_1997_frobenius_manifolds} proved that the generic genus-zero Riemann--Hilbert correspondence is \emph{symplectic},
	upon restriction to the symplectic leaves of~\eqref{eq:tame_betti}---%
	which involve $G$-conjugacy classes at the marked points.}
\vspace{5pt}

In recent years, several works have extended many features of 2d gauge theory to the case of connections with \emph{irregular/wild} singularities, cf.~(among many others)~\cite{jimbo_miwa_ueno_1981_monodromy_preserving_deformation_of_linear_ordinary_differential_equations_with_rational_coefficients_i_general_theory_and_tau_function,
	malgrange_1983_sur_les_deformations_isomonodromiques_II_singularites_irregulieres,malgrange_1991_equations_differentielles_a_coefficients_polynomiaux,martinet_ramis_1991_elementary_acceleration_and_multisummability_I,kashiwara_1998_semisimple_holonomic_d_modules,malgrange_2004_deformations_isomonodromiques_forme_de_liouville_fonction_tau}.
This led to the discovery of new algebraic Poisson varieties generalising~\eqref{eq:tame_betti}, the so-called  ``wild'' character varieties~\cite{boalch_2001_symplectic_manifolds_and_isomonodromic_deformations,boalch_2014_geometry_and_braiding_of_stokes_data_fission_and_wild_character_varieties,boalch_yamakawa_twisted_wild_character_varieties},
parameterising refined monodromy data known as Stokes data~\cite{stokes_1857_on_the_discontinuity_of_arbitrary_constants_which_appear_in_divergent_developments,birkhoff_1913_the_generalized_riemann_problem_for_linear_differential_equations_and_the_allied_problems_for_linear_difference_and_q_difference_equations,
	balser_jurkat_lutz_1979_birkhoff_invariants_and_stokes_multipliers_for_meromorphic_linear_differential_equations,sibuya_1990_linear_differential_equations_in_the_complex_domain_problems_of_analytic_continuation}---%
etc.;
cf.~\cite{boalch_2021_topology_of_the_stokes_phenomenon} for a modern (topological) view.

What matters here is that wild character varieties are also equipped with algebraic Poisson automorphisms.
As above, such automorphisms arise from fundamental groups of suitable manifolds $\bm B$ carrying families of wild character varieties. A key point, which we will explain in more detail below, is that a new ``local braiding'' appears in the wild setting: namely, one has interesting group actions even for manifolds $\bm B$ carrying constant families of pointed Riemann surfaces.
This leads to actions of $\mf g$-\emph{braid groups}~\cite{boalch_2002_g_bundles_isomonodromy_and_quantum_weyl_groups,boalch_2014_geometry_and_braiding_of_stokes_data_fission_and_wild_character_varieties},
which in the pure case are fundamental groups of root-hyperplane complements for the Lie algebra $\mf g = \Lie(G)$.\footnote{
	These are the ``generalised'' braid groups à la Brieskorn/Deligne~\cite{brieskorn_1971_die_fundamentalgruppe_des_raumes_der_regulaeren_orbits_einer_endlichen_komplexen_spiegelungsgruppe,brieskorn_1973_sur_les_groupes_de_tresses_d_apres_v_i_arnold,deligne_1972_les_immeubles_des_groupes_de_tresses_generalises}, a.k.a.~the Artin(--Tits) groups of type $\mf g$~\cite{tits_1969_le_probleme_des_mots_dans_les_groupes_de_coxeter,brieskorn_saito_1972_artin_gruppen_und_coxeter_gruppen} (cf.~\S~\ref{sec:notions_notations}).}

Part of this story was given a quantum field theory interpretation by Witten~\cite{witten_2008_gauge_theory_and_wild_ramification}, and moreover the symplectic leaves of these Poisson varieties were shown to be new (complete) hyperkähler manifolds in the Sabbah/Biquard--Boalch extension of the nonabelian Hodge correspondence on Riemann surfaces~\cite{sabbah_1999_harmonic_metrics_and_connections_with_irregular_singularities,biquard_boalch_2004_wild_non_abelian_hodge_theory_on_curves}---%
cf.~\cite{huang_sun_2023_meromorphic_parahoric_higgs_torsors_and_filtered_stokes_g_local_systems_on_curves} for general structure group.

\vspace{5pt}

The viewpoint discussed above subsumes many previous examples of nonlinear monodromy actions in 2d gauge theory, including:
\begin{enumerate}
	\item $\bm B = \Conf_m(\mb C)$, the genus-zero tame case, leading to the action of the pure braid group $\PB_m$ on the (tame) character variety. This is the first example to have been investigated, dating back to Hurwitz's work \cite{hurwitz_1891_ueber_riemannsche_flaechen_mit_gegebenen_verzweigungspunkten};

	\item $\bm B = \mc M_g$, the nonsingular genus-$g$ case, leading to the action of the mapping class group $\Gamma_g$ of $\Sigma$ (cf.~\cite{birman_1974_braids_links_and_mapping_class_groups} for the relation with braid groups);

	\item and $\bm B = \mf t_{\reg}$, related to irregular connections with a ``generic'' fixed pole of order two in genus zero,\footnote{
		      \emph{Generic} deformations involve regular semisimple leading terms; in type $A$, these are diagonal matrices with distinct entries. Cf.~\cite{balser_jurkat_lutz_1979_birkhoff_invariants_and_stokes_multipliers_for_meromorphic_linear_differential_equations,jimbo_miwa_ueno_1981_monodromy_preserving_deformation_of_linear_ordinary_differential_equations_with_rational_coefficients_i_general_theory_and_tau_function,malgrange_1983_sur_les_deformations_isomonodromiques_II_singularites_irregulieres}, developing the subject started in~\cite{birkhoff_1913_the_generalized_riemann_problem_for_linear_differential_equations_and_the_allied_problems_for_linear_difference_and_q_difference_equations}, and cf.~Exmp.~\ref{ex:generic_case}.}~where $\mf t_{\reg}$ is the regular part of a Cartan subalgebra $\mf t \sse \mf g$. This leads to the action of the pure $\mf g$-braid group $\PB_{\mf g}$  on $G^*$~\cite{boalch_2002_g_bundles_isomonodromy_and_quantum_weyl_groups} (viz.,
	      the semiclassical action of the quantum Weyl group of De Concini--Kac--Procesi~\cite{deconcini_kac_procesi_1992_quantum_coadjoint_action}).
\end{enumerate}

The latter case is a first instance of the new braiding mentioned above, having to do with the structure group of the connections:
not with the motion of poles on the surface, nor with the deformation of the complex structure of the surface.
(E.g.,
if $G = \GL_n(\mb C)$ we can deform the noncoalescing eigenvalues of an $n$-by-$n$ diagonal matrix.)
In this simplest wild example, one can in fact switch back to the tame case of moving points on $\Sigma$, via the Fourier--Laplace transform (a.k.a.~Harnad's duality~\cite{harnad_1994_dual_isomonodromic_deformations_and_moment_maps_to_loop_algebras}, cf.~\cite{yamakawa_2011_middle_convolution_and_harnad_duality,boalch_2012_simply_laced_isomonodromy_systems}); however, this is not possible in general.

Most of the present paper is concerned with fundamental groups of spaces carrying families of wild character varieties for \emph{nongeneric} irregular types. Going beyond the generic case leads to more complicated “multi-scale” braiding phenomena, the new braidings of~\cite{boalch_2014_geometry_and_braiding_of_stokes_data_fission_and_wild_character_varieties}. Our goal is to describe explicitly how these braidings look like, and in particular make precise the natural intuition that they amount to cabling of braids (an idea also mentioned by Ramis~\cite{ramis_2012_iso_irregular_deformations_of_linear_ode_and_dynamics_of_painleve_equations}).
This phenomenon is first visible when studying deformations of nongeneric connections with poles of order three, as in~\cite{boalch_2012_simply_laced_isomonodromy_systems} (cf.~\S~\ref{sec:ex-typeA}).

\subsubsection*{Wild Riemann surfaces}

The starting point is Boalch's definition of “wild” Riemann surfaces~\cite[Def.~8.1]{boalch_2014_geometry_and_braiding_of_stokes_data_fission_and_wild_character_varieties}, generalising Riemann surfaces with distinct marked points  (cf.~\S~\ref{sec:deformations}). Wild Riemann surfaces are endowed with “irregular types” at the marked points, governing the irregular part of the (formal) normal form of meromorphic connections.
Precisely, recall that in a local trivialization around a point $a \in \Sigma$ a meromorphic connection on a principal $G$-bundle over $\Sigma$ is encoded by a $\mf g$-valued meromorphic 1-form $\mc{A}$ on $\Sigma$, where $\mf g$ is the Lie algebra of $G$.
We assume that up to a local holomorphic gauge transformation,
and up to holomorphic terms,
one has
\begin{equation}
	\label{eq:local_form}
	\mc{A} = \dif Q + \frac{\Lambda}{z} \dif z,
\end{equation}
where $z$ is a local coordinate with $z(a) = 0$, and where
\begin{equation*}
	\Lambda \in \mf g, \qquad Q = \sum_{i = 1}^p A_i z^{-i} \in z^{-1} \mf t [z^{-1}],
\end{equation*}
for a Cartan subalgebra $\mf t \sse \mf g$ and an integer $p \geq 1$.
Then $Q$ is the irregular type at the point $a \in \Sigma$, corresponding to a ``very good'' gauge orbit (and the formal residue $\Lambda$ can be taken to commute with $Q$,
cf.~\cite{boalch_2018_wild_character_varieties_meromorphic_hitchin_systems_and_dynkin_diagrams}).
This means that in this paper we only consider untwisted/unramified irregular types.\fn{
	The twisted/ramified case,
including the admissible deformations of the underlying irregular classes,
has now been considered in~\cite{doucot_rembado_2025_topology_of_irregular_isomonodromy_times_on_a_fixed_pointed_curve,boalch_doucot_rembado_2025_twisted_local_wild_mapping_class_groups_configuration_spaces_fission_trees_and_complex_braids,doucot_rembado_yamakawa_twisted_local_g_wild_mapping_class_groups}.}

\vspace{5pt}

Furthermore, there is a natural notion of an ``admissible family'' of wild Riemann surfaces (cf.~Def. \ref{def:admissible_deformations}), roughly ensuring that the irregular types on each fibre are such that the unipotent Stokes automorphisms have the ``same shape''.
Deforming irregular types subject to this admissibility condition, one obtains spaces whose fundamental groups yield the new braiding actions on wild character varieties. Indeed, the result~\cite[Thm.~10.2]{boalch_2014_geometry_and_braiding_of_stokes_data_fission_and_wild_character_varieties} (extending~\cite{boalch_2001_symplectic_manifolds_and_isomonodromic_deformations,boalch_2002_g_bundles_isomonodromy_and_quantum_weyl_groups,boalch_2007_quasi_hamiltonian_geometry_of_meromorphic_connections} from the generic case) shows that any admissible family of wild Riemann surfaces over a base space $\bm B$ determines a (nonlinear) fibre bundle $\ul{\mc M}_{\on{B}} \to \bm B$ of Poisson wild character varieties, equipped with a complete flat Ehresmann connection (the \emph{wild} nonabelian Gauß--Manin connection~\cite{boalch_2001_symplectic_manifolds_and_isomonodromic_deformations}).
Finally,
just as in the tame case,
the fundamental group $\pi_1(\bm B)$ acts by algebraic Poisson automorphisms on any fibre.
Upon restricting to a symplectic leaf,
this is the wild analogue of the ``symplectic nature'' of $\pi_1(\Sigma^{\circ})$~\cite{goldman_1984_the_symplectic_nature_of_fundamental_groups_of_surfaces} (cf.~\cite{deligne_malgrange_ramis_2007_singularites_irregulieres}).

\subsubsection*{Goal of the paper}

The discussion in the previous paragraph involves the choice of an admissible family of wild Riemann surfaces over a base $\bm B$. If one wants to focus on the new phenomena arising from the variation of irregular types, it is natural to consider a constant family of pointed Riemann surfaces, only deforming the irregular types at the marked points. Such admissible families are parameterised by a universal deformation space (cf.~\S~\ref{sec:deformations}), whose fundamental group we call a (pure) \emph{local wild} mapping class group. The main aim of this paper is to give a description of local wild mapping class groups which:
(i) works for arbitrary reductive groups and for deformations of nongeneric irregular types;
but (ii) is concrete and explicit enough to be useful for studying algebraic Poisson dynamics on wild character varieties (cf.~Exmp.~\ref{ex:braiding_stokes_data}). In particular, we will study the aforementioned ``braiding of braids'' (cf.~Thm.~\ref{thm: 4-intro} and the discussion preceding it).

\subsection{Main results and layout of the paper}
\label{sec:main_results}

We now outline the structure and main results of this paper,
and refer to the body of the document for more precise statements.
Let $(\mf g,\mf t)$ be a finite-dimensional split reductive Lie algebra defined over $\mb C$, and let $\bm{\Sigma} = (\Sigma, a, Q)$ be a one-pointed\footnote{Cf.~Rmk.~\ref{rem:many_point_case} for the many-point case.}~wild Riemann surface.

Following~\cite{boalch_2014_geometry_and_braiding_of_stokes_data_fission_and_wild_character_varieties},
in \S~\ref{sec:deformations} we define a universal space of admissible deformations $\bm B_Q$ of the irregular type $Q$ (keeping the underlying pointed Riemann surface $(\Sigma,a)$ fixed;
cf.~Def.~\ref{def:deformation_space}).
We then define the (pure) \emph{local wild mapping class group} (WMCG, cf.~\cite[\S~8]{boalch_2014_poisson_varieties_from_riemann_surfaces}) $\Gamma_Q$ to be the fundamental group of the space of admissible deformations, cf.~\S~\ref{sec:local_WMCG}.
The notion of admissible deformations depends on the root system $\Phi_{\mf g} = \Phi(\mf g,\mf t)$, and on the pole orders $d_{\alpha} \in \set{0,\dotsc,p}$ of the meromorphic function germ $q_{\alpha} = \alpha \circ Q$, for all $\alpha \in \Phi_{\mf g}$: for a deformation to be admissible, one asks that the pole orders after evaluation at each $\alpha$ remain constantly equal to $d_\alpha$. One of the main aims of this paper is to describe $\Gamma_Q$---which in principle depends on $Q$, $p$ and the integers $d_\alpha$---in terms of the couple $(\Phi_{\mf g}, \mf t)$.

We will see (cf.~\eqref{eq:pure_local_WMCG_factors}) that $\Gamma_Q$ breaks into a product
\begin{equation}
	\label{eq:pure_local_WMCG_factors-intro}
	\Gamma_Q \simeq \prod_{i = 1}^p \pi_1(\bm B_i,A_i),
\end{equation}
of fundamental groups of certain hyperplane complements arising as deformation spaces of each of the coefficients of $Q$. This simple observation will be the starting point to construct, in \S~\ref{sec:fission}, an increasing filtration of root subsystems of $\Phi_{\mf g}$, obtained by \emph{fission}, which will be our key tool to describe $\Gamma_Q$. In the generic case this filtration is trivial and one obtains the pure $\mf g$-braid group; in general, we find a generalisation thereof controlled by a nested sequence of (Levi) root subsystems. See \S~\ref{sec:ex-typeA} for a concrete illustration of this phenomenon.

In \S~\ref{sec:general_results} we prove some general properties of local wild mapping class groups.
First we explain how their description can be reduced to the case of simple Lie algebras, cf.~\S~\ref{sec:simple_is_enough}.
Secondly, we obtain a uniform bound on the number of nontrivial factors of the local WMCG in~\eqref{eq:pure_local_WMCG_factors-intro}.

\begin{theorem}[cf.~\S~\ref{sec:descending_ranks}]
	The number of nontrivial factors of~\eqref{eq:pure_local_WMCG_factors-intro} is at most the semisimple rank of $\mf g$ (i.e., the rank of $\Phi_{\mf g}$).
\end{theorem}
Note that the bound in the previous theorem is independent of the order $p$ of the pole of the irregular type, and only depends on $\mf g$. This bound---which is not evident a priori---rests on the description of the local WMCG via fission.

Finally, as another application of fission we classify local WMCGs for low-rank Lie algebras.

\begin{theorem}[cf.~\S~\ref{sec:low_rank}]
	\leavevmode
	\begin{enumerate}
		\item If the semisimple rank of $\mf g$ is one, then the local WMCG is either trivial or infinite cyclic (i.e., isomorphic to the pure $\mf g$-braid group).

		\item If the semisimple rank of $\mf g$ is two, then the local WMCG is either trivial or isomorphic to one of the groups $\mb{Z}$, $\mb{Z}^2$, or the pure $\mf g$-braid group.
	\end{enumerate}
\end{theorem}

In particular this classifies the local WMCGs for the exceptional simple Lie algebra of type $G_2$, while from \S~\ref{sec:type_A} we focus on \emph{classical} simple Lie algebras, and give a complete explicit description of the local WMCG.

Beginning with type $A$, we attach a tree to any sequence of root subsystems obtained from fission, which we thus call a \emph{fission tree} (see Def.~\ref{def:fission_tree} and cf.~\cite[App.~C]{boalch_irregular_connections_and_kac_moody_root_systems}). An example of this construction is given in \S~\ref{sec:ex-typeA}, to which we refer the reader for an illustration of the next theorem.
Similarly, in \S~\ref{sec:type_BC}, we attach a \emph{bichromatic} tree to any irregular type of type $B/C$ (cf.~Def.~\ref{def:bichromatic_fission_tree}); and finally a generalisation thereof in type $D$, in \S~\ref{sec:type_D} (cf.~Def.~\ref{def:generalised_tree}).
This leads to the following description of local wild mapping class groups, for \emph{all} classical simple Lie algebras.

\begin{theorem}[cf.~Thmm.~\ref{thm:tree_gives_wmcg_type_A},~\ref{thm:tree_gives_wmcg_type_BC} and~\ref{thm:tree_gives_wmcg_type_D}]
	The generalised fission tree uniquely determines the local WMCG, as follows: at each node of the tree one attaches the pure braid group of an explicit hyperplane arrangement, and the local WMCG is the product of those factors.
\end{theorem}

Let us point out that for types $A$ and $B/C$ all factors correspond to root-hyperplane arrangements, while in type $D$ there is a unique noncrystallographic factor,
further studied in Prop.~\ref{thm:exotic_type}.
(The simplest example involves seven hyperplanes in $\mb C^3$,
cf.~Rmkk.~\ref{rmk:noncrystallography} and~\ref{rmk:triality}.)

Finally, in \S~\ref{sec:braid_operad} we relate the (monochromatic) fission trees of type $A$ to \emph{cabled} braids---as in the title. This formalises the driving idea that local wild mapping class groups ought to be described in terms of multi-scale braiding, cf.~the example of \S~\ref{sec:ex-typeA}.
More precisely, to any tree $\mc T$ we attach a pure \emph{cabled} braid group $\ms{P\!B}(\mc T)$ by using the compositions of the pure braid group operad (cf.~Def.~\ref{def:pure_cabled_artin_braid_group}), and we prove the following.

\begin{theorem}[cf.~Thm.~\ref{thm:wmcg_type_A_equal_cabled_braid_groups}]
	\label{thm: 4-intro}
	If $\mc T$ is the fission tree of a type-$A$ irregular type $Q$, then there is a group isomorphism $\Gamma_Q \simeq \ms{P\!B}(\mc T)$.
\end{theorem}

To sum up, our results show that $\Gamma_Q$ is obtained via the following steps (the last one being performed via operads only in type $A$; a conjectural generalisation is given in Conj.~\ref{conj:generalised_operads}):
\begin{equation*}
	\mf g \xra{\text{fission}} \Phi_1 \sse \dm \sse \Phi_{p+1} = \Phi_{\mf g} \xra{\text{fission tree}} \mc T  \xra{\text{braid cabling}} \ms{P\!B}(\mc T).
\end{equation*}

Let us finally note that the use of fission trees goes much beyond a description of pure local wild mapping class groups: they have since also be used to study the full/nonpure case~\cite{doucot_rembado_2025_topology_of_irregular_isomonodromy_times_on_a_fixed_pointed_curve}, also twisted~\cite{boalch_doucot_rembado_2025_twisted_local_wild_mapping_class_groups_configuration_spaces_fission_trees_and_complex_braids,doucot_rembado_yamakawa_twisted_local_g_wild_mapping_class_groups}, and they have applications beyond the local case (cf.~particularly~\cite[\S\S~3.7+5]{boalch_doucot_rembado_2025_twisted_local_wild_mapping_class_groups_configuration_spaces_fission_trees_and_complex_braids}).

\subsubsection*{Conventions}

In this text all Lie algebras, commutative (associative, unitary) algebras, and tensor products,
are tacitly defined over $\mb C$.
A few basic notions/notations, used throughout the body of the paper, are summarised in App.~\ref{sec:notions_notations}.

\subsection{An example in type A}
\label{sec:ex-typeA}

To conclude this introduction, we explicitly work out an example in type $A$ illustrating the main phenomena which will be studied later on. This serves both as a motivation and as a test case for the general results.

Fix again a pointed Riemann surface $(\Sigma,a)$, and choose a local coordinate $z$ vanishing at the marked point.
In the case of meromorphic connection $\nabla$ on a holomorphic rank-$n$ vector bundle $\mc E \to \Sigma$, an (untwisted) irregular type $Q$ at $a$ is an element $Q \in z^{-1}\mf t[z^{-1}]$, where $\mf t \sse \mf{gl}_n(\mb C)$ is the standard Cartan subalgebra of diagonal matrices.
Consider in particular the irregular type
\begin{equation*}
	Q = \frac{A_2}{z^2} + \frac{A_1}z, \qquad A_1, A_2 \in \mf t,
\end{equation*}
corresponding to a pole of order 3 for $\nabla$ at the marked point.
Such an irregular type is generic when the leading coefficient $A_2$ has $n$ distinct eigenvalues (see~\cite{jimbo_miwa_ueno_1981_monodromy_preserving_deformation_of_linear_ordinary_differential_equations_with_rational_coefficients_i_general_theory_and_tau_function,malgrange_1983_sur_les_deformations_isomonodromiques_II_singularites_irregulieres}).	We are rather interested in the nongeneric situation:
suppose that $Q$ is \emph{not} generic.

\subsubsection*{Admissible deformations}

Deforming the irregular type $Q$ means varying the coefficients of $A_2$ and $A_1$.
What does it mean for such a deformation to be admissible?

It goes as follows.
First decompose the fibre $E_a \simeq \mb C^n$ of the vector bundle over $a \in \Sigma$ into eigenspaces for $A_2 \in \End_{\mb C}(E_a)$:
\begin{equation}
	\label{eq:decomposition_fibre_1}
	\mb C^n = \bops_{i = 1}^k V_i, \qquad V_i = \Ker(A_2 - \lambda_i \Id_{\mb C^n}) \sse \mb C^n,
\end{equation}
where $\set{ \lambda_1,\dc,\lambda_k } \sse \mb C$ is the spectrum of $A_2$.
Then use $[A_2,A_1] = 0$, whence $A_1 (V_i) \sse V_i$, and split further into eigenspaces for the semisimple restrictions $A_{1,i} \ceqq \eval[1]{A_1}_{V_i} \in \End_{\mb C}(V_i)$:
\begin{equation}
	\label{eq:decomposition_fibre_2}
	V_i = \bops_{j = 1}^{k_i} W_{ij}, \qquad W_{ij} = \Ker( A_{1,i} - \lambda_{j,i} \Id_{V_i}) \sse V_i,
\end{equation}
where $\set{ \lambda_{1,i},\dc,\lambda_{k_i,i} } \sse \mb C$ is the spectrum of $A_{1,i}$.
An \emph{admissible} deformation of $(A_1,A_2) \in \mf t^2$ is another pair $(A_1',A_2')$ inducing the same decompositions~\eqref{eq:decomposition_fibre_1}--\eqref{eq:decomposition_fibre_2}.
It follows that the space of admissible deformations $\bm B_Q$ will depend on the multiplicities of the eigenvalues of $A_1$ and $A_2$.
For arbitrary reductive Lie algebras, we will generalise this looking at the positions of the coefficients of irregular types relative to the root-hyperplanes of $(\mf g,\mf t)$.

\subsubsection*{Local WMCG and cabling of braids}

Keeping the notation from \S~\ref{sec:ex-typeA}, we will now give an example of braid cabling to describe elements of the pure local wild mapping class group $\pi_1(\bm B_Q)$.

Concretely, consider first a loop in the configuration space of the (ordered) eigenvalues of $A_2 \in \End(\mb C^n)$, keeping them distinct, which yields a (pure) braid $\sigma \in \PB_k$; then the $i$-th strand of $\sigma$ can be replaced by another braid $\tau_i \in \PB_{k_i}$, corresponding to braiding the eigenvalues of $A_{1,i} \in \End(V_i)$ for $i \in \set{1,\dc,k}$.
The result of this operation is a cabled braid
\begin{equation}
	\label{eq:cabling}
	(\sigma,\tau_1,\dc,\tau_k) \lmt \gamma(\sigma;\tau_1,\dc,\tau_k) \in \PB_{\bm k}, \qquad \bm k = \sum_i k_i,
\end{equation}
as in the theory of (group/action) operads, cf.~\S~\ref{sec:braid_operad} and~\cite[Chp.~5]{yau_2021_infinity_operads_and_monoidal_categories_with_group_equivariance}.

Let us now give the simplest nontrivial concrete example: take $n = 3$ and consider the irregular type with coefficients
\begin{equation}
	\label{eq:example_irregular_type}
	A_2 =
	\begin{pmatrix}
		1 &   &   \\
		  & 1 &   \\
		  &   & 0
	\end{pmatrix}, \quad A_1 =
	\begin{pmatrix}
		-1 &   &   \\
		   & 1 &   \\
		   &   & 0
	\end{pmatrix} \in \mf{gl}_3(\mb C).
\end{equation}
Its admissible deformations look like
\begin{equation*}
	A_2' =
	\begin{pmatrix}
		a &   &    \\
		  & a &    \\
		  &   & a'
	\end{pmatrix}, \quad A_1' =
	\begin{pmatrix}
		b &    &   \\
		  & b' &   \\
		  &    & c
	\end{pmatrix},
\end{equation*}
with $a,a',b,b',c \in \mb C$ such that $a \neq a'$ and $b \neq b'$.
(Up to homotopy equivalence,
one could equivalently take coefficients/deformations inside $\mf{sl}_3(\mb C) \sse \mf{gl}_3(\mb C)$.)

With the above notation, we have $(k;k_1,k_2) = (2;2,1)$, and we can,
e.g.,
take pure braids as in Fig.~\ref{fig:braids_to_be_cabled}.
Here we read braid diagrams bottom-to-top,
and note that $\sigma \neq \tau_1$.
(Colours are for comparison with Fig.~\ref{fig:cabling}.)
Then $\tau_2 \in \PB_1$ is necessarily trivial, and the resulting pure cabled braid is as in Fig.~\ref{fig:cabling}.
The strands 1 and 2 in Fig.~\ref{fig:cabling} move in parallel, and are first braided with the strand 3: this corresponds to the (distinct) eigenvalues $(a,a')$ looping around each other, following the braid $\sigma$; then the (distinct) eigenvalues $(b,b')$ do the same, following $\tau_1$, so the strands 1 and 2 are also eventually braided---on the top left corner of the diagram.
The latter phenomenon is possible because one eigenspace of the leading coefficient $A_2$ is 2-dimensional, and it breaks in the subleading coefficient $A_1$.

\begin{figure}
	\begin{center}
		\begin{tikzpicture}
			\pic[
				name prefix=braid,
				braid/strand 1/.style={red!50!black},
				braid/strand 2/.style={green!50!black}
			]
			at (0,0) {braid={s_1s_1}};
			\node[at=(braid-2-1),label=west:{$\sigma =$}] {};
			\node[at=(braid-1-1),label=east:{\, ,}] {};
		\end{tikzpicture} \quad
		\begin{tikzpicture}
			\pic[
				name prefix=braid,
			]
			at (0,0) {braid={s_1^{-1}s_1^{-1}}};
			\node[at=(braid-2-1),label=west:{$\tau_1 =$}] {};
			\node[at=(braid-1-1),label=east:{$\in \PB_2.$}] {};
		\end{tikzpicture}
	\end{center}
	\caption{Before cabling
		\label{fig:braids_to_be_cabled}}
\end{figure}

\begin{figure}
	\begin{center}
		\begin{tikzpicture}
			\pic[
				name prefix=braid,
				braid/strand 1/.style={red!50!black},
				braid/strand 2/.style={red!50!black},
				braid/strand 3/.style={green!50!black},
			]
			at (0,0) {braid={s_1^{-1}s_1^{-1}s_2s_1s_1s_2}};
			\node[at=(braid-1-3),label=west:{$\gamma(\sigma;\tau_1,\tau_2) = $}] {};
			\node[at=(braid-2-3),label=east:{$\in \PB_3.$}] {};
			\node[at=(braid-1-e),label=south:{$1$}] {};
			\node[at=(braid-2-e),label=south:{$2$}] {};
			\node[at=(braid-3-e),label=south:{$3$}] {};
		\end{tikzpicture}
	\end{center}
	\caption{Example of (pure) braid cabling
		\label{fig:cabling}}
\end{figure}

\subsubsection*{Local WMCG and trees}

Going back to the general setting of this section, encoding the operation~\eqref{eq:cabling} in terms of its inputs/outputs naturally yields trees.
Namely, the splitting~\eqref{eq:decomposition_fibre_1} can be pictorially represented as in Fig.~\ref{fig:node_fission}.
This diagram is just a (planar) tree rooted at the node $\ast$, with a choice of ordering for the leaves---%
which are all at the same level.

\begin{figure}
	\begin{center}
		\begin{tikzpicture}
			\node at (0,.5) {$\mc T_{A_2} =$};
			\foreach \name/\x/\y in {A1/1/0,A2/2/0,A4/4/0,A5/5/0,
					B1/3/1}
			\vertex (\name) at(\x,\y){};
			\node[label=south:1] at (1,0) {};
			\node[label=south:2] at (2,0) {};
			\node[label=south:$k-1$] at (4,0) {};
			\node[label=south:$k$] at (5,0) {};
			\node[label=north:$\ast$] at (3,1) {};
			\node at (3,0) {$\dm$};
			\foreach \from/\to in {A1/B1,A2/B1,A4/B1,A5/B1}
			\draw (\from) -- (\to);
		\end{tikzpicture}
	\end{center}
	\caption{One-step fission
		\label{fig:node_fission}}
\end{figure}

The same can be done for the splittings~\eqref{eq:decomposition_fibre_2}, getting trees $\mc T_{A_{1,1}}, \dc, \mc T_{A_{1,k}}$.
Gluing each of them at the corresponding leaf of $\mc T_{A_2}$ then yields a tree $\mc T_{A_2,A_1}$ of height 2;
e.g.,
the irregular type~\eqref{eq:example_irregular_type} yields Fig.~\ref{fig:example_fission_tree}.

\begin{figure}
	\begin{center}
		\begin{tikzpicture}
			\node[label=north:$\ast$] at (2.25,2) {};
			\node at (0,1) {$\mc T_{A_2,A_1} =$};
			\foreach \name/\x/\y in {A1/1/0,A2/2/0,A3/3/0,
					B1/1.5/1,B2/3/1,
					C1/2.25/2}
			\vertex (\name) at(\x,\y){};
			\node[label=south:1] at (1,0) {};
			\node[label=south:2] at (2,0) {};
			\node[label=south:3] at (3,0) {};
			\foreach \from/\to in {A1/B1,A2/B1,A3/B2,
					B1/C1,B2/C1}
			\draw (\from) -- (\to);
		\end{tikzpicture}
	\end{center}
	\caption{Example of type-$A$ fission tree
		\label{fig:example_fission_tree}}
\end{figure}

Iterating this procedure associates a tree $\mc T_Q$ with any irregular type $Q$, cf.~\cite[App.~C]{boalch_irregular_connections_and_kac_moody_root_systems} and Exmp.~\ref{ex:examples_presentations}.
Finally, the pure braid group operad can be evaluated on $\mc T_Q$, and the result is a group isomorphic to the pure local WMCG $\Gamma_Q$.
For instance, the tree in Fig.~\ref{fig:example_fission_tree} yields a group isomorphic to $\PB_2 \times \PB_2$: its generators are the elements $(\sigma,\tau_1)$ used in the cabling of Fig.~\ref{fig:cabling}.
One can use them to compute the braiding of Stokes data (cf.~Exmp.~\ref{ex:braiding_stokes_data}). In general, this procedure formalises the intuition of a ``multi-scale braiding'' description of local WMCGs.

\subsubsection*{The general case}

Let us highlight some of the main differences between the general setting considered in the body of the paper and the example discussed above.
\begin{enumerate}
	\item We consider principal bundles with arbitrary (reductive) structure groups $G$. We do not use faithful $G$-modules to define admissible deformations; analogously, we will not assume that $\mf g$ is a matrix Lie algebra. This is necessary in order to work with general structure groups, cf.~\cite[Lem~A.2]{boalch_2002_g_bundles_isomonodromy_and_quantum_weyl_groups}.

	\item We shall use the root system of $(\mf g,\mf t)$ to construct deformation spaces of irregular types, generalising (differences of) eigenvalues of semisimple endomorphisms of $\mb C^n$; these are root valuation strata,
	      which were considered, for different purposes, in~\cite{goresky_kottwitz_macpherson_2009_codimensions_of_root_valuation_strata}.

	\item We will introduce more general \emph{fission} trees for Lie algebras of classical types $B/C$ and $D$: to do so we will have to deal with several possible ``fissions'' at each stage, related to the irreducible components of Levi subsystems of the root systems at hand, and we will accordingly introduce new decorations on the trees to encode them. In particular,
	      this leads to a more precise version of the multilevel braiding conjecture, for any classical Lie algebra: see Conj.~\ref{conj:generalised_operads}.
\end{enumerate}

\renewcommand{\thetheorem}{\arabic{section}.\arabic{theorem}} %

\subsection*{Acknowledgements}

We thank P.~Boalch, F.~Naef and G.~Paolini for listening/answering to some of our questions.

\section{Admissible deformations of wild Riemann surfaces}
\label{sec:deformations}

We start by recalling the general definition of admissible families of wild Riemann surfaces in the untwisted setting, following~\cite{boalch_2014_geometry_and_braiding_of_stokes_data_fission_and_wild_character_varieties}.

\vspace{5pt}

Consider:
(i) a closed Riemann surface $\Sigma$;
(ii) a connected complex reductive Lie group $G$;
(iii) its Lie algebra $\mf g = \Lie(G)$;
(iv) a maximal torus $T \sse G$;
and (v) the corresponding Cartan subalgebra $\mf t = \Lie(T) \sse \mf g$.
A (dressed, untwisted) wild Riemann surface structure on $\Sigma$ is the choice of a finite \emph{ordered} set $\bm a = (a_1,\dc, a_m) \in \Sigma^m$ of $m \geq 0$ distinct marked points, and \emph{untwisted} irregular types $\bm Q = (Q_1,\dc,Q_m)$ based there; in turn an untwisted irregular type is the germ of a $\mf t$-valued meromorphic function, defined up to holomorphic terms.

Precisely,
let $\ms O_{\Sigma,a}$ be the local ring of germs of functions at a point $a \in \Sigma$, $\wh{\ms O}_{\Sigma,a}$ its completion, and $\wh{\ms K}_{\Sigma,a}$ the field of fractions of this latter. Consider the quotient $\ms T_{\Sigma,a} \ceqq \wh{\ms K}_{\Sigma,a} \bs \wh{\ms O}_{\Sigma,a}$, consisting of ``tails'' of formal Laurent series. By definition, an untwisted irregular type at $a \in \Sigma$ is an element
\begin{equation*}
	Q \in \mf t \ots \ms T_{\Sigma,a}.
\end{equation*}

If $z$ is a local coordinate with $z(a) = 0$ then
\begin{equation*}
	\wh{\ms O}_{\Sigma,a} \simeq \mb C \llb z \rrb, \qquad \wh{\ms K}_{\Sigma,a} \simeq \mb C (\!(z)\!),
\end{equation*}
and $\ms T_{\Sigma,a} \simeq \mb C (\!(z)\!) \bs \mb C \llb z \rrb$, so we can write
\begin{equation}
	\label{eq:untwisted_irregular_type}
	Q = \sum_{i = 1}^p A_i z^{-i} \in z^{-1} \mf t [z^{-1}] \simeq \mf t (\!( z ) \!) \bs \mf t \llb z \rrb, \qquad A_1,\dc,A_p \in \mf t,
\end{equation}
for some integer $p \geq 1$.
Hereafter we simply refer to such elements as ``irregular types''---all implicitly untwisted/unramified.

\begin{remark}
	Again,
	irregular types single out (and give intrinsic meaning to) the time-variables for isomonodromic deformations of irregular-singular meromorphic connections on principal $G$-bundles over $\Sigma$.
	Notably, they control normal forms for (the principal of) local connection 1-forms, as in~\eqref{eq:local_form}.

	More precisely,
	the element~\eqref{eq:untwisted_irregular_type} controls the exponential factors of the fundamental solutions of the associated system of \emph{linear} differential equations (i.e.,
	the local horizontal sections of the connection), which ultimately lead to the Stokes data encoding their exponential growth/decay along prescribed directions at the pole.
	As mentioned in \S~\ref{sec:background},
	this is the starting point to define deformations of the meromorphic connections which are \emph{isomonodromic}, i.e.,
	such that Stokes data are kept locally constant.
\end{remark}

Let $\bm \Sigma = (\Sigma,\bm a,\bm Q)$ be a wild Riemann surface, which we want to deform in \emph{admissible} fashion.
To this end,
let $\bm B$ be a complex manifold and
\begin{equation*}
	\Sigma_b = \pi^{-1}(b) \lhra \ul \Sigma \lxra{\pi} \bm B, \qquad b \in \bm B,
\end{equation*}
a holomorphic family of closed Riemann surfaces fibering over $\bm B$ (with smooth fibres,
of constant genus).
Choose an $m$-tuple
\begin{equation*}
	\ul{\bm a} = (\ul a_1,\dc,\ul a_m) \cl \bm B \lra \ul \Sigma^m,
\end{equation*}
of global sections.
Finally,
consider a holomorphic $\bm B$-family of irregular types\footnote{
	Later on the points will be fixed and there will be a \emph{fixed} target space of irregular types,
	whence $\ul{\bm Q}$ will be a holomorphic map with domain $\bm B$;
	cf.~\cite{doucot_rembado_tamiozzo_moduli_spaces_of_untwisted_wild_riemann_surfaces} in general.}~$b \mt \ul Q_i(b)$, based at the marked points $\ul a_i(b) \in \Sigma_b$, and let $\ul{\bm Q} = \bigl( \ul Q_1,\dc,\ul Q_m \bigr)$ be their collection.

\begin{definition}
	The triple $(\ul \Sigma,\ul{\bm a},\ul{\bm Q})$ is a holomorphic $\bm B$-\emph{family} of wild Riemann surfaces.
	Denoting it by $\ul{\bm \Sigma} \to \bm B$, the ``fibre'' at $b \in \bm B$ is the wild Riemann surface
	\begin{equation*}
		\bm \Sigma_b = \bigl( \Sigma_b,\bm a_b,\bm Q_b \bigr), \qquad \bm a_b \ceqq \bigl( \ul a_1(b),\dc,\ul a_m(b) \bigr), \quad \bm Q_b \ceqq \bigl( \ul Q_1(b),\dc,\ul Q_m(b) \bigr).
	\end{equation*}

	If $0 \in \bm B$ is a base point, and $\bm B$ is connected, the family $\ul{\bm \Sigma} \to (B,0)$ is a \emph{deformation} of the ``starting'' wild Riemann surface $\bm \Sigma_0$.
\end{definition}

To introduce the admissibility condition let $\Phi_{\mf g} = \Phi(\mf g,\mf t) \sse \mf t^{\dual}$ be the root system of the split Lie algebra $(\mf g,\mf t)$, and consider an irregular type as in~\eqref{eq:untwisted_irregular_type}.
Then introduce the meromorphic function germ (defined up to holomorphic terms)
\begin{equation*}
	q_{\alpha} \ceqq \alpha (Q)	\in \ms T_{\Sigma,a}, \qquad \alpha \in \Phi_{\mf g},
\end{equation*}
obtained by evaluating the irregular type on a root.\fn{
	This means evaluating $\alpha \ots \Id \cl \mf t \ots \ms T_{\Sigma,a} \to \mb C \ots \ms T_{\Sigma,a} \simeq \ms T_{\Sigma,a}$ at $Q$.}~Hence a collection of $\bm B$-families of irregular types yields $\bm B$-families
\begin{equation*}
	b \lmt \ul q_{j,\alpha}(b) \ceqq \alpha \circ \bigl( \ul Q_j(b) \bigr) \in \ms T_{\Sigma_b,\ul a_j(b)},
\end{equation*}
for $j \in \set{1,\dc,m}$ and $\alpha \in \Phi_{\mf g}$.

Finally, if $q$ is the germ of a meromorphic function at $a \in \Sigma$, let $\ord(q) \in \mb Z_{\geq 0}$ be its pole order at the base point, with the convention that $\ord(q) \ceqq 0$ if $q$ is holomorphic.
(This integer remains well-defined up to adding holomorphic terms.)

\begin{definition}[\cite{boalch_2014_geometry_and_braiding_of_stokes_data_fission_and_wild_character_varieties}, Def.~10.1]
	\label{def:admissible_deformations}
	An \emph{admissible deformation} of $\bm \Sigma_0$ is a deformation $\ul{\bm \Sigma} \to (\bm B,0)$ of $\bm \Sigma_0$ such that for all $b \in \bm B$:
	\begin{enumerate}
		\item  the marked points $\bm a_b \in (\Sigma_b)^m$ are distinct;

		\item and one has
		      \begin{equation}
			      \label{eq:admissible_deformations}
			      \ord \bigl( \, \ul q_{\alpha,j}(b) \bigr) = \ord \bigl( \, \ul q_{\alpha,j}(0) \bigr) \in \mb Z_{\geq 0}, \qquad \text{for all } \alpha \in \Phi_{\mf g}.
		      \end{equation}
	\end{enumerate}
\end{definition}

Therefore,
the cardinality of the sets of marked points and the pole orders of the irregular types evaluated at each root are \emph{constant} along the deformation.

Recall the set of nonzero pole orders which occur at $\ul a_i(b) \in \Sigma_b$ is the set of \emph{levels} of the irregular type $\ul Q_i(b)$: this paper is essentially about the multilevel case.

\subsection{Wild deformations}

By Def.~\ref{def:admissible_deformations} we can deform the complex structure of the Riemann surface underlying a wild one, and move the marked points inside their configuration space.
These are the tame isomonodromy times (controlled by the stack $\mc M_{g,m}$,
where $g$ is the genus of $\Sigma_0$); as explained in \S~\ref{sec:background}, we are rather interested in the additional local wild moduli of the irregular types---freezing the underlying pointed surface.

\vspace{5pt}

Consider thus again an irregular type as in~\eqref{eq:untwisted_irregular_type}.
Introducing the local coordinate $x = z^{-1}$ (on a punctured neighbourhood of the marked point) yields a polynomial
\begin{equation}
	\label{eq:base_irregular_type}
	Q = \sum_{i = 1}^p A_i x^i \in x \, \mf t [x], \qquad A_i \in \mf t.
\end{equation}
To avoid discussing trivial cases, assume $A_p \neq 0$ in what follows.

Now $q_{\alpha} \in x\mb C[x] \cup \set{0}$ for all roots $\alpha \in \Phi_{\mf g}$, and the pole orders are controlled by the function
\begin{equation*}
	\bm d_Q \cl \alpha \lmt d_{\alpha} = \deg_x (q_{\alpha}), \qquad \alpha \in \Phi_{\mf g},
\end{equation*}
with the analogous convention that $\deg_x(0) = 0$.
(On the whole $\bm d_Q(\Phi_{\mf g}) \sse \Set{0,\dc,p}$, and $d_{\alpha} = 0$ if an only if $q_{\alpha} = 0$.)

Thus~\eqref{eq:admissible_deformations} becomes
\begin{equation}
	\label{eq:admissible_deformations_2}
	\deg_x(\alpha \circ Q') = d_{\alpha}, \qquad \alpha \in \Phi_{\mf g},
\end{equation}
where $Q' = \sum_{i > 0} A_i' x^i$ is any other polynomial with coefficients $A_i' \in \mf t$.

\begin{remark}
	In principle we should consider polynomials of arbitrary degree, but only those with $\deg_x(Q') \leq p$ will contribute to the topology of the (universal) space of admissible deformations.

	Indeed imposing~\eqref{eq:admissible_deformations_2} yields $A'_i \in \Ker(\Phi_{\mf g}) \sse \mf t$ for $i > p$, so all coefficients of higher degree live in a contractible space.
	(Note that $\Ker(\Phi_{\mf g}) = \bigcap_{\Phi_{\mf g}} \Ker(\alpha) = \mf Z_{\mf g}$, which is \emph{not} zero in the nonsemisimple case.)
\end{remark}

Hence we pose the following:

\begin{definition}[Cf.~\cite{boalch_2014_geometry_and_braiding_of_stokes_data_fission_and_wild_character_varieties}, Exmp.~10.1]
	\label{def:deformation_space}
	The \emph{universal deformation space} of~\eqref{eq:base_irregular_type} is
	\begin{equation}
		\label{eq:deformation_space}
		\bm B_Q \ceqq \Set{ Q' \in x \, \mf t [x] | \deg_x(Q') \leq p, \, \deg_x( \alpha \circ Q') = d_{\alpha} \text{ for } \alpha \in \Phi_{\mf g} }.
	\end{equation}
\end{definition}

This depends on $Q$ via the integer $p \geq 1$ and the tuple $\bm d_Q \in \mb Z_{\geq 0}^{\Phi_{\mf g}}$; the use of the term ``universal'' is justified in Rmk.~\ref{rem:moduli_spaces}.

To describe $\bm B_Q$ let us use the natural identification
\begin{equation*}
	\Set{ Q' \in x \, \mf t [x] | \deg_x(Q') \leq p } \lxra{\simeq} \mf t^p,
\end{equation*}
mapping
\begin{equation}
	\label{eq:identification}
	Q' = \sum_{i = 1}^p A_i' x^i \lmt (A_1',\dc,A_p') \in \mf t^p.
\end{equation}
We obtain an inclusion $\bm B_Q \sse \mf t^p$,
by which we endow $\bm B_Q$ with the subspace topology (of the analytic topology),
and the tuple $\bm A = (A_1,\dc,A_p) \in \bm B_Q$ corresponds to the base point~\eqref{eq:base_irregular_type}.

\begin{proposition}
	\label{prop:decomposition_deformation_space}
	There is a product decomposition $\bm B_Q = \prod_{i = 1}^p \bm B_i\sse \mf t^p$, where
	\begin{equation}
		\label{eq:deformation_space_for_coefficients}
		\bm B_i \ceqq \bigcap_{d_{\alpha} < i} \Ker(\alpha) \cap \bigcap_{d_{\alpha} = i} \bigl( \mf t \sm \Ker(\alpha) \bigr) \sse \mf t.
	\end{equation}
\end{proposition}

\begin{proof}
	By definition $(A_1',\dc,A_p') \in \bm B_Q$ if and only if for any root $\alpha \in \Phi_{\mf g}$ one has
	\begin{equation*}
		\alpha(A'_{d_{\alpha}}) \neq 0, \qquad \alpha(A'_i) = 0, \qquad i > d_{\alpha},
	\end{equation*}
	and there are no conditions on $A_1',\dc,A'_{d_{\alpha}-1} \in \mf t$.
	Hence
	\begin{equation*}
		\bm B_Q = \bigcap_{\alpha \in \Phi_{\mf g}} \left( \Biggl( \, \prod_{1 \leq i < d_{\alpha}} \mf t \Biggr) \times \bigl( \mf t \sm \Ker(\alpha) \bigr) \times \prod_{d_{\alpha} < i  \leq p} \Ker(\alpha) \right) \sse \mf t^p,
	\end{equation*}
	and the conclusion follows by swapping products/intersections.
\end{proof}

We conclude this section with few observations.

\begin{remark}[Universal deformation]
	\label{rem:moduli_spaces}

	The pointed space $(\bm B_Q, \bm A)$ is a fine moduli space of admissible deformations of the irregular type $Q$ on the ``starting'' wild Riemann surface $\bm \Sigma = (\Sigma, a, Q)$.

	Indeed, let $\bm B$ be a connected pointed complex manifold, and consider the constant family of Riemann surfaces $\ul \Sigma \ceqq \Sigma \times \bm B \to \bm B$, with the constant section $\ul a = (a, \Id) \cl \bm B \to \Sigma \times \bm B$.
	The projection on the first factor of $\Sigma \times \bm B$ induces isomorphisms $\Sigma_b \simeq \Sigma$ and $\ms T_{\Sigma_b, \ul a(b)} \simeq \ms T_{\Sigma, a}$ for all $b \in  \bm B$.
	Let $\ms T_{\Sigma, a}^{\leq p} \sse \ms T_{\Sigma, a}$ be the subspace of tails of Laurent series of pole order at most $p$. Via the previous identification, a holomorphic $\bm B$-family of irregular types $\ul Q$ of pole order at most $p$ at $\ul a$ is a holomorphic function $\bm B \to \mf t \otimes \ms T_{\Sigma, a}^{\leq p}$.
	Let us identify as above $\mf t \otimes  \ms T_{\Sigma, a}^{\leq p}$ with $\mf t^p$ by sending $\sum_{i = 1}^p A_i' z^{-i}$ to $(A_1', \ldots, A_p')$.
	Then, by construction, a map of pointed complex manifolds $f: (\bm B, 0) \to (\mf t^p, \bm A)$ yields an admissible deformation $(\Sigma \times \bm B, \ul a, \ul Q)$ of $(\Sigma, a, Q)$ if and only if $f(\bm B) \sse \bm B_Q$.
	By definition, this means that  $(B_Q, \bm A)$ is the fine moduli space of admissible deformations of $Q$ (on the fixed pointed Riemann surface $(\Sigma, a)$) with pole order bounded by $p$.
	In particular, the family of irregular types
	\begin{equation*}
		\ul Q \cl \bm B_Q \lra \mf t \otimes \ms T_{\Sigma, a}^{\leq p}, \qquad (A_1', \ldots, A_p') \longmapsto  \sum_{i = 1}^p A'_i z^{-i},
	\end{equation*}
	is the universal admissible deformation of $Q$ with pole order bounded by $p$.
\end{remark}

\begin{remark}[Trivial deformations]
	\label{rem:trivial_deformations}

	One can add an element of the centre to any coefficient,
	i.e.,
	the product $(\mf Z_{\mf g})^p$ acts on $\bm B_Q$ by factorwise translations.

	Hence in principle one could consider the quotient space $\bm B_Q \bs \mf Z_{\mf g}^p$, which yields the same fundamental group, and amounts to considering the semisimple part of $\mf g$: this will be done later on, but at this stage the main definitions are cleaner without such restrictions.
	Moreover,
	one needs reductive Lie algebras to discuss fission recursively (cf.~\S~\ref{sec:fission}).
\end{remark}

\begin{remark}[Intrinsic definition]
	The pole order of---%
	the germ of---%
	a meromorphic function on $\Sigma$ is well-defined up to local (formal) biholomorphisms.
	Therefore,
	the integers $d_{\alpha} \in \set{ 0, \dc, p }$ depend on $Q$ only, and not on the identifications $\wh{\ms O}_{\Sigma,a} \simeq \mb C \llb z \rrb \sse \mb C (\!( z )\!) \simeq \wh{\ms K}_{\Sigma,a}$: the space~\eqref{eq:deformation_space} is thus well defined by the intrinsic discrete valuations of the completed local ring/field of $\Sigma$ at $a$ (cf.~\cite{goresky_kottwitz_macpherson_2009_codimensions_of_root_valuation_strata,doucot_rembado_tamiozzo_moduli_spaces_of_untwisted_wild_riemann_surfaces}).

	Conversely,
	choosing a uniformizer is helpful to give explicit descriptions---%
	and factorizations---%
	of the deformation spaces,
	and changing uniformizer yields homeomorphisms thereof.
\end{remark}

\begin{remark}[Many-point case]
	\label{rem:many_point_case}
	It is straightforward to extend~\eqref{eq:deformation_space} to the case of several \emph{fixed} marked points on $\Sigma$.

	Namely,
	choosing marked points $\bm a = (a_1,\dc,a_m) \in \Sigma^m$ we still consider a trivial family $\ul \Sigma = \Sigma \times \bm B_{\bm Q} \to \bm B_{\bm Q}$, equipped with the corresponding global constant sections, and this time $\bm B_{\bm Q}$ is a space of (simultaneous) admissible deformations of irregular types at each marked point.
	More precisely $\bm Q = (Q_1,\dc,Q_m)$, with
	\begin{equation*}
		Q_j \in \mf t \ots \ms T_{\Sigma,a_j}, \qquad j \in \set{1,\dc,m},
	\end{equation*}
	each with a pole of order $p_j\geq 0$, and then
	\begin{equation*}
		\bm B_{\bm Q} \ceqq \prod_{j = 1}^m \bm B_{Q_j} \sse \prod_{j = 1}^m \mf t^{p_j}. \qedhere
	\end{equation*}
\end{remark}

\section{Pure local wild mapping class groups}
\label{sec:local_WMCG}

The main definition is the following:

\begin{definition}
	\label{def:wild_mapping_class_group}

	The \emph{pure local wild mapping class group} (WMCG) of the wild Riemann surface $\bm \Sigma = (\Sigma,a,Q)$ is
	\begin{equation}
		\label{eq:pure_local_WMCG}
		\Gamma_Q \ceqq \pi_1(\bm B_Q,\bm A),
	\end{equation}
	where as before $\bm A = (A_1,\dc,A_p) \in \mf t^p$ and $Q = \sum_{i = 1}^p A_i x^i$.
	(We will also say that $\Gamma_Q$ is a pure local WMCG \emph{of type} $\mf g$.)
\end{definition}

Once more, this does not depend on the underlying pointed surface $(\Sigma,a)$, but only on (the integers associated with) $Q$.
The terminology is chosen with a view towards the \emph{global} WMCG, to be defined elsewhere (cf.~\S~\ref{sec:outlook}).

Importantly,
by Prop.~\ref{prop:decomposition_deformation_space} there is a product decomposition
\begin{equation}
	\label{eq:pure_local_WMCG_factors}
	\Gamma_Q \simeq \prod_{i = 1}^p \pi_1(\bm B_i,A_i),
\end{equation}
and moreover the many-point case yields a product of such groups, with a factor at each marked point, as in Rmk.~\ref{rem:many_point_case}.

In the rest of the paper we will study the pure local WMCGs, particularly aiming at a classification of (the isomorphism class of) the factors~\eqref{eq:pure_local_WMCG_factors}.

\begin{example}[Abelian case]
	\label{rmk:abelian_case}

	Suppose that $\mf g$ is abelian.
	Then $\Phi_{\mf g} = \vn$, and $\bm B_Q = \mf t^p$: all pure local WMCGs of type $\mf g$ are trivial.
\end{example}

\begin{example}[Unique pole order]
	\label{ex:generic_case}

	Suppose that $\bm d$ is \emph{constant}, say $d_{\alpha} = d \in \set{ 0,\dc,p }$ for all $\alpha \in \Phi_{\mf g}$.
	If $d = 0$ then $Q \in z^{-1}\mf Z_{\mf g}[z^{-1}]$, and~\eqref{eq:deformation_space} is the contractible space $\mf Z_{\mf g}^p \sse \mf t^p$: $\Gamma_Q$ is trivial.

	Else $d > 0$, which is then the unique level of $Q$,
	and the generic case corresponds to $d = p$.
	One finds
	\begin{equation*}
		\bm B_i = \mf t, \quad \bm B_d = \mf t_{\reg},\quad \bm B_{i'} = \mf Z_{\mf g},
	\end{equation*}
	for $i < d < i'$.
	Hence $\bm B_d$ is a strong deformation retract of $\bm B_Q$, and
	\begin{equation*}
		\Gamma_Q \simeq \pi_1(\mf t_{\reg},A_d) = \PB_{\mf g},
	\end{equation*}
	recovering the pure $\mf g$-braid group.
\end{example}

As explained in the introduction, in the next sections we will describe the topology of the deformation space in the case where several levels occur, going \emph{beyond} pure $\mf g$-braid groups.

\section{Filtrations and fission}
\label{sec:fission}

In this section we will rewrite~\eqref{eq:deformation_space_for_coefficients} as the complement of a hyperplane arrangement, involving the root system $\Phi_{\mf g}$ in an essential way.
This makes it possible to prove the general results of \S~\ref{sec:general_results}.
Moreover we will introduce Dynkin diagrams, which will be used for the classification statements of \S\S~\ref{sec:type_A}--\ref{sec:type_D}.

\vspace{5pt}

The starting point is noticing that there is an increasing sequence of subsets
\begin{equation}
	\label{eq:root_filtration}
	\Phi_1 \sse \dm \sse \Phi_{p+1} = \Phi_{\mf g}, \qquad \text{with} \qquad \Phi_i \ceqq \Set{ \alpha \in \Phi_{\mf g} | d_{\alpha} < i}.
\end{equation}
(In particular $\Phi_1 = \Set{ \alpha \in \Phi_{\mf g} | q_{\alpha} = 0 }$.)

But there is more structure, which follows from the triangular inequality of the standard valuation of the field of formal Laurent series in one variable:

\begin{lemma}
	Every term of~\eqref{eq:root_filtration} is a root subsystem of $\Phi_{\mf g}$.
\end{lemma}

A different proof follows from the fact that each subset $\Phi_i \sse \Phi_{\mf g}$ is the root system of a reductive Lie subalgebra $\mf h_i \sse \mf g$ (containing $\mf t$).

\begin{remark}
	Different filtrations~\eqref{eq:root_filtration} may yield isomorphic pure local WMCGs of type $\mf g$.
	E.g.,
	acting on each term by the group $\Aut(\Phi_{\mf g})$ of automorphism of the root system---or in particular by the Weyl group.
	In this paper we shall use some such identifications to treat the classification problem.
\end{remark}

In particular we can now rewrite~\eqref{eq:deformation_space_for_coefficients} as
\begin{equation}
	\label{eq:deformation_space_for_coefficients_2}
	\bm B_i = \Ker(\Phi_i) \cap \bigcap_{\Phi_{i+1} \sm \Phi_i} \bigl( \mf t \sm \Ker(\alpha) \bigr) \sse \mf t,
\end{equation}
so in principle the factors of $\Gamma_Q$ are controlled by nested root subsystems of $\Phi_{\mf g}$.
However not all subsystems will appear, but rather only \emph{Levi subsystems}: these arise by taking nested centralisers, as we will momentarily explain.

\subsection{Fission: Lie groups/algebras}

The sequence~\eqref{eq:root_filtration} is associated with a filtration of complex reductive subgroups of $G$ (cf.~\cite[Eq.~33]{boalch_2014_geometry_and_braiding_of_stokes_data_fission_and_wild_character_varieties}).
In turn their Lie algebras are (reductive) Levi factors of parabolic subalgebras of $\mf g$, which we use to give a more explicit description of~\eqref{eq:deformation_space_for_coefficients_2}.

\vspace{5pt}

Let us then define the ``fission'' subgroups
\begin{equation*}
	H_i \ceqq \Set{ g \in G | \Ad_g (A_k) = A_k, \, i \leq k \leq p } \sse G, \qquad i \in \set{1,\dc,p},
\end{equation*}
fitting into an increasing sequence $H_1 \sse \dm \sse H_p \sse G$ of connected complex reductive groups.

\begin{remark}
	As mentioned in the introduction, the terminology is due to the ``breaking'' of the structure group of the principal bundle at the boundary of the (real, oriented) blowup of $(\Sigma,a)$, from $G$ down to $H_1$.
	This phenomenon is only visible in the wild case, and it is different from the usual ``fusion'' operation (= sewing surfaces with boundaries, along their boundaries).
\end{remark}

In particular
\begin{equation*}
	H_1 = \Set{ g \in G | \Ad_g(Q) = Q }
\end{equation*}
is the centraliser of the irregular type in $G$,
viz.,
the stabiliser for the diagonal Adjoint action $G \to \GL(\mf g)$ on each coefficient.
Note that $T \sse H_1$, and we allow for strict inclusions.
(One says that fission is "complete" if $T = H_1$,
which includes the generic case.)

\begin{example}
	In the setting of Exmp.~\ref{ex:generic_case} we find
	\begin{equation*}
		H_i = T \ , \qquad  H_i' = G,
	\end{equation*}
	for $i \leq d < i'$:
	the structure group breaks down to the maximal torus as soon as the regular semisimple coefficient is encountered.
\end{example}

Denote now by $\mf h_i \ceqq \Lie(H_i)$ the $i$-th ``fission'' subalgebra, which is the centraliser of the coefficients $A_i,\dc,A_p \in \mf t$ in $\mf g$.
In particular
\begin{equation*}
	\mf h_1 = \Set{ X \in \mf g | [X,Q] = 0 }
\end{equation*}
is the centraliser of $Q$ in $\mf g$.
As expected $\mf h_1$ contains $\mf t = \Lie(T)$, and in turn $\mf t \sse \mf h_i$ is a Cartan subalgebra for $i \in \set{1,\dc,p}$.

\begin{lemma}
	\label{lem:root_subsystems}
	One has $\Phi_i = \Phi_{\mf h_i} \ceqq \Phi(\mf h_i,\mf t)$ for $i \in \set{1,\dc,p}$, in the notation of~\eqref{eq:root_filtration}.
\end{lemma}

\begin{proof}
	By induction on $i \in \set{p,\dc,1}$.
	The base is the identity
	\begin{equation}
		\label{eq:root_system_levi}
		\Phi_{\mf h_p} = \Phi_{\mf g} \cap \set{A_p}^{\perp} \sse \Phi_{\mf g},
	\end{equation}
	which follows by observing that $\mf g_{\alpha} \cap \mf h_p \neq (0)$ if and only if $\alpha(A_p) = 0$.

	Then replacing $(H_p,G,A_p)$ with $(H_i,H_{i+1}, A_i)$ at each step proves the claim.
\end{proof}

Analogous ``descending'' inductions will be a common theme in the rest of the paper.
In particular here let us consider the centraliser $\mf h = \Ker(\ad_A) \sse \mf g$ of an element $A \in \mf t$ (in $\mf g$).

\begin{remark}[Reductive centralisers]
	\label{rem:reductive_centraliser}
	It is important here that $\mf h$ is reductive.

	Beware however it need \emph{not} be (semi)simple, even if $\mf g$ is: e.g.,
	if $A \in \mf t_{\reg}$ then $\mf h = \mf t$ is even abelian.
\end{remark}

Write $\Phi_{\mf h} = \Phi(\mf h,\mf t) \sse \Phi_{\mf g}$, which is the subset of roots vanishing on $A$---as in~\eqref{eq:root_system_levi}.
Then, up to repeating all constructions by replacing $\mf g$ with $\mf h$ (and keeping $\mf t$), to understand~\eqref{eq:deformation_space_for_coefficients_2} it is enough to study the topological space
\begin{equation}
	\label{eq:true_complement}
	\bm B(\Phi_{\mf h},\Phi_{\mf g}) \ceqq \Ker(\Phi_{\mf h}) \bigsm \bigcup_{\Phi_{\mf g} \sm \Phi_{\mf h}} \Ker(\alpha) \sse \mf t.
\end{equation}

Later we will show that~\eqref{eq:true_complement} is never empty, so indeed it is a hyperplane complement: it is obtained by ``restricting'' the hyperplane arrangement of $\Phi_{\mf g} \sm \Phi_{\mf h}$ to $\Ker(\Phi_{\mf h}) \sse \mf t$.
Note that this generalises $\mf t_{\reg}$, which in turn corresponds to the generic case $\Phi_{\mf h} = \vn$.
In particular we do \emph{not} expect that the factors of $\Gamma_Q$ will be pure braid groups of Lie algebras, and indeed there is a counterexample in type $D$ (cf.~\S~\ref{sec:type_D}).

\begin{remark}[Dimensions]
	One has
	\begin{equation}
		\label{eq:dimension_kernel}
		\dim_{\mb C} \bigl( \Ker(\Phi_{\mf h}) \bigr) = \rk(\mf g) - \rk(\Phi_{\mf h}) = \rk(\mf g) - \rk(\mf h'),
	\end{equation}
	where $\mf h' = [\mf h,\mf h] \sse \mf h$ is the semisimple part.
\end{remark}

\subsection{Fission: Dynkin diagrams}

Let again $\mf h \sse \mf g$ be the centraliser of an element $A \in \mf t$.
Then $\Phi_{\mf h} \sse \Phi_{\mf g}$ is a Levi subsystem, and it is well-known that for all such there exists a base of simple roots $\Delta_{\mf g} \sse \Phi_{\mf g}$ such that $\Phi_{\mf h}$ corresponds to a subdiagrams of the Dynkin diagram of $\Phi_{\mf g}$.

Namely, consider the following subspace of the complex plane:
\begin{equation}
	\mf C = \Set{ \lambda \in \mb C | \on{Re}(\lambda) \geq 0 \, ; \text{ if } \on{Re}(\lambda) = 0 \text{ then } \on{Im}(\lambda) \geq 0 }.
\end{equation}
The identities $\mf C \cup (-\mf C) = \mb C$ and $\mf C \cap (-\mf C) = \set{0}$ are ``complexifications'' of the analogous one for the subspace $\mb R_{\geq 0} \sse \mb R$.
Then one can prove that the following subspace is a fundamental domain for the action of the Weyl group, for any choice of a base $\Delta_{\mf g} \sse \Phi_{\mf g}$:
\begin{equation}
	\mf C_{\Delta_{\mf g}} = \Set{ A' \in \mf t | \Braket{ \theta | A' } \in \mf C \text{ for all } \theta \in \Delta_{\mf g} } \sse \mf t.
\end{equation}
The latter is thus a natural ``complexification'' of the $\Delta_{\mf g}$-dominant Weyl chamber---in the real form of the Cartan subalgebra determined by the roots~\cite{collingwood_mcgovern_1993_nilpotent_orbits_in_semisimple_lie_algebras}.
(Note that by definition $\mf Z_{\mf g} \sse \mf C_{\Delta_{\mf g}}$, and the Weyl group acts trivially thereon.)

Hence, up to acting via the Weyl group on the choice of base, we can assume that $A \in \mf C_{\Delta_{\mf g}}$.
One can now prove that the subset $\Delta_{\mf h} \ceqq \Delta_{\mf g} \cap \set{A}^{\perp} \sse \Phi_{\mf h}$ is a base for $\Phi_{\mf h}$.
This is equivalent to proving that $\mf h$ is the Levi factor of the (standard) parabolic subalgebra of $\mf g$ corresponding to $\Delta_{\mf h} \sse \Delta_{\mf g}$, cf.~\cite[Prop.~5.6]{crooks_2019_complex_adjoint_orbits_in_lie_theory_and_geometry}.

Hence $\Phi_{\mf h}$ admits a base given by the simple roots of $\Phi_{\mf g}$ which vanish on $A$.

\begin{remark}
	One can rewrite~\eqref{eq:true_complement} using $\Ker(\Phi_{\mf h}) = \Ker(\Delta_{\mf h})$. (This is the Lie-centre of $\mf h$.)

	On the contrary, it is \emph{not} enough to remove the hyperplanes of $\Delta_{\mf g} \sm \Delta_{\mf h}$; rather those of $\Phi_{\mf g}^+ \sm \Phi_{\mf h}^+ \sse \Phi_{\mf g} \sm \Phi_{\mf h}$, where $\Phi_{\mf h}^+ = \Phi_{\mf g}^+ \cap \set{A}^{\perp} \sse \Phi_{\mf h}$,
	and in turn as usual $\Phi_{\mf g}^+ = \Phi_{\mf g} \cap \spann_{\mb Z_{>0}} (\Delta_{\mf g})$ (the former is a system of positive roots for $(\mf h,\mf t)$).
\end{remark}

Thus,
the Dynkin diagram $\mc D_{\mf h}$ of $(\Phi_{\mf h},\Delta_{\mf h})$ is obtained by choosing a subset of nodes of the Dynkin diagram $\mc D_{\mf g}$ of $(\Phi_{\mf g},\Delta_{\mf g})$, keeping all edges among them (and their decorations, i.e.,
possible doubling/tripling and orientation).
Repeating this procedure at each step, as in~\cite[Lem.~3.2.5]{calaque_felder_rembado_wentworth_wild_orbits_and_generalised_singularity_modules_stratifications_and_quantisation}, finally yields a nested sequence of Dynkin diagrams:
\begin{equation}
	\label{eq:sequence_dynkin}
	\mc D_{\mf h_1} \sse \dm \sse \mc D_{\mf h_p} \sse \mc D_{\mf g}.
\end{equation}
In this viewpoint, 
``fission'' roughly corresponds to the way in which a connected component of $\mc D_{\mf h_{i+1}}$ breaks into connected components of $\mc D_{\mf h_i}$,
and this leads to the classification of \S\S~\ref{sec:type_A}--\ref{sec:type_D}.
In brief,
denote as customary by $\bullet_n$ the irreducible rank-$n$ root systems of the simple Lie algebras of classical type,
with $\bullet \in \set{A,B,C,D}$.
Let then $\mc D_{\bullet_n}$ be their Dynkin diagrams with respect to the standard bases.
Now loosely speaking:
\begin{itemize}
	\item all components of a Dynkin subdiagram $\mc D \sse \mc D_{A_n}$ are of type $A$;

	\item at most \emph{one} component of a Dynkin subdiagram $\mc D \sse \mc D_{B_n}$ (resp.,
	      $\mc D \sse \mc D_{C_n}$, $\mc D \sse \mc D_{D_n}$) is of type $B$ (resp.,
	      $C$, $D$), and the others are of type $A$.
\end{itemize}

We will make this precise (cf.~Rmkk.~\ref{rmk:B_component} and~\ref{rmk:about_D_components}),
thereby encoding a sequence such as~\eqref{eq:sequence_dynkin} into a decorated tree which generalises the examples of \S~\ref{sec:ex-typeA}.

\section{General results}
\label{sec:general_results}

Before jumping into the classification,
however,
we prove a few abstract results---%
building on the material of the previous section.

\subsection{Hyperplane arrangements}

Given any root subsystem $\Phi \sse \Phi_{\mf g}$,
consider the vector subspace $U \ceqq \Ker(\Phi) \sse \mf t$.
It is natural to ask whether the ``restricted'' hyperplane arrangement
\begin{equation}
	\label{eq:restricted_arrangement_general}
	\mc{H} = \Set{ \Ker(\alpha) \cap U = \Ker\bigl( \eval[1]{\alpha}_U \bigr)
		| \alpha \in \Phi_{\mf g} \sm \Phi } \sse \mb P \bigl( U^{\dual} \bigr)
\end{equation}
is crystallographic, i.e.,
if it comes from a root system: but this is false (cf.~Rmk.~\ref{rmk:noncrystallography}).
The basic obstruction of course is that $\Phi_{\mf g} \sm \Phi \sse \Phi_{\mf g}$ is \emph{not} a root (sub)system in general.

\begin{remark}
	Still writing $U = \Ker(\Phi)$, note that the set
	\begin{equation}
		\label{eq:restricted_root_subsystem}
		\eval[1]{\Phi_{\mf g}}_U = \Set{\eval[1]{\alpha}_U | \alpha \in \Phi_{\mf g} } \sse U^{\dual},
	\end{equation}
	of restricted roots,
	is naturally identified with the quotient set $\Phi_{\mf g} \bs U^{\perp} \sse \mf t^{\vee} \bs U^{\perp}$.
	In turn,
	one has
	\begin{equation*}
		U^{\perp} = \bigl(\Ker(\Phi)\bigr)^{\!\perp} = \spann_{\mb C}(\Phi) \sse \mf t^{\dual},
	\end{equation*}
	and so this is the same as considering $\Phi_{\mf g} \bs \spann_{\mb C}(\Phi) \sse \mf t^{\dual} \bs \spann_{\mb C}(\Phi)$.
\end{remark}

\begin{remark}
	Studying the reflection group of~\eqref{eq:restricted_arrangement_general} goes towards the \emph{full/nonpure} local WMCGs, which will be studied elsewhere (cf.~\S~\ref{sec:outlook};
	this is subtler than restricting the reflections associated with $\alpha \in \Phi_{\mf g} \sm \Phi$ to $U$, even in type $A$.)
\end{remark}

One of the insights of this work is that such restrictions/quotients of root systems, and their hyperplane arrangements, naturally arise in the theory of isomonodromic deformations for wild connections on principal bundles over Riemann surfaces.

\subsection{Reduction to the simple case}
\label{sec:simple_is_enough}

Suppose that there is a direct decomposition $\mf g = \bops_i \mf I_i$ into (mutually-commuting) ideals $\mf I_i \sse \mf g$, and let $\mf t_i = \mf t \cap \mf I_i$---%
which is a Cartan subalgebra of $\mf I_i$.
There is then a second decomposition $\mf t = \bops_i \mf t_i$, which induces an analogous one on the dual $\mf t^{\dual} \simeq \bops^{\perp}_i \mf t_i^{\dual}$, by identifying $\mf t_i^{\dual}$ with the subspace
\begin{equation*}
	\bigcap_{j \neq i} \mf t_j^{\perp} = \bigl( \mf t \ominus \mf t_i \bigr)^{\perp} \sse \mf t^{\dual}, \qquad \mf t \ominus \mf t_i \ceqq \bops_{j \neq i} \mf t_j.
\end{equation*}

Write now $\Phi_{\mf I_i} = \Phi(\mf I_i,\mf t_i)$, so that there is a splitting of root systems
\begin{equation}
	\label{eq:decomposition_root_system}
	(\mf t,\Phi_{\mf g}) \simeq \bops_i (\mf I_i,\Phi_{\mf I_i}).
\end{equation}

Finally, for any root subsystem $\Phi \sse \Phi_{\mf g}$ set
\begin{equation*}
	\Phi^{(i)} \ceqq \Phi \cap \mf t_i^{\dual} \sse \Phi,
\end{equation*}
finding a disjoint union $\Phi = \coprod_i \Phi^{(i)}$.
Then the subset $\Phi^{(i)} \sse \Phi_{\mf I_i} $ is also a root subsystem~\cite[Ch.~VI, \S~1.2]{bourbaki_1968_groupes_et_algebres_de_lie_chapitres_iv_v_vi}.

\begin{proposition}
	In the notation of~\eqref{eq:true_complement}, there is a product decomposition
	\begin{equation*}
		\bm B(\Phi,\Phi_{\mf g}) = \prod_i \bm B \bigl( \Phi^{(i)},\Phi_{\mf I_i} \bigr) \sse \mf t.
	\end{equation*}
\end{proposition}

\begin{proof}
	If $\alpha \in \Phi_{\mf I_i} \sse \Phi_{\mf g}$ then
	\begin{equation*}
		\Ker(\alpha) = (\mf t \ominus \mf t_i) \ops \Ker(\alpha_i) \sse \mf t, \qquad \alpha_i \ceqq \eval[1]{\alpha}_{\mf I_i} \in \mf t_i^{\dual}.
	\end{equation*}
	Intersecting along the partition $\Phi = \coprod_i \Phi^{(i)}$ then yields
	\begin{align*}
		\Ker(\Phi) & =
		\prod_i \Biggl( \bigcap_{\Phi^{(i)}} \Ker(\alpha_i) \Biggr) \sse \prod_i \mf t_i = \mf t.
	\end{align*}

	Analogously,
	one has
	\begin{equation*}
		\bigcap_{\Phi_{\mf g} \sm \Phi} \bigl( \mf t \sm \Ker(\alpha) \bigr)
		= \prod_i \Biggl( \bigcap_{\Phi_{\mf I_i} \sm \Phi^{(i)}} \bigl( \mf t_i \sm \Ker(\alpha_i) \bigr) \Biggr) \sse \prod_i \mf t_i,
	\end{equation*}
	and the statement follows by intersecting the two.
\end{proof}

Applying this to a filtration of root subsystems yields a product decomposition for pure local WMCGs of type $\mf g$, into pure local WMCGs of type $\mf I_i$.
In particular the splitting $\mf g = \mf Z_{\mf g} \ops \bops_i \mf I_i$, of a reductive Lie algebra into its simple ideals and its centre, means it is enough to work with \emph{simple} Lie algebras---%
recall that $\Gamma_Q$ is trivial in the abelian case (cf.~Rmk.~\ref{rmk:abelian_case}).

\vspace{5pt}

\begin{center}
	\textbf{Hereafter we will thus assume $\mf g$ to be simple.}
\end{center}

\subsection{Nonempty complements}

Choose again $A \in \mf t$, and let $\mf h \sse \mf g$ be its centraliser.

\begin{lemma}
	\label{lem:nonempty_complement}
	The complement~\eqref{eq:true_complement} is nonempty.
\end{lemma}

\begin{proof}
	Choose that $\beta \in \Phi_{\mf g} \sm \Phi_{\mf h}$, and suppose that $\Ker(\Phi_{\mf h}) \sse \Ker(\beta)$---%
	by contradiction.
	The latter happens if and only if $\mb C \cdot \beta \sse \spann_{\mb C}( \Phi_{\mf h} )$, which implies $\beta(A) = 0$: this is absurd, as $\Phi_{\mf h} \sse \Phi_{\mf g}$ is the subset of roots vanishing on $A$.
\end{proof}

\begin{remark}
	The statement of Lem.~\ref{lem:nonempty_complement} is false for general root subsystems $\Phi \sse \Phi_{\mf g}$.
	E.g.,
	the subsystem of short/long roots inside the root system of type $G_2$ yields an empty complement.
	This corresponds to the proper inclusion $A_2 \sse G_2$, which in turn does \emph{not} correspond to an inclusion of the (finite) Dynkin diagrams.

	Again, the point is that here we have Levi subsystems, i.e.,
	the inclusion $\Phi_{\mf h} \sse \spann_{\mb C}(\Phi_{\mf h}) \cap \Phi_{\mf g}$ is an equality.
\end{remark}

\subsection{Descending ranks}
\label{sec:descending_ranks}

If $\rk{\mf h'} = \rk{\mf g}$, it follows by~\eqref{eq:dimension_kernel} that the complement~\eqref{eq:true_complement} is homotopically trivial---and nonempty by Lem.~\ref{lem:nonempty_complement}.
Hence to have a nontrivial fundamental group we need the rank to diminish at each step, and we find that:
\begin{corollary}
	The number of nontrivial factors of~\eqref{eq:pure_local_WMCG_factors} is at most $\rk(\mf g)$
\end{corollary}

Note that this bound is independent of the pole order $p \geq 1$ of $Q$, and of the pole orders $d_{\alpha}$ of the functions $q_{\alpha} = \alpha \circ Q$.

\subsection{Low-rank cases}
\label{sec:low_rank}

Suppose furthermore that $\rk(\mf g) - \rk(\mf h') = 1$.
Then the space $\Ker(\Phi_{\mf h}) \sse \mf t$ is a complex line, and since the relative complement cannot be empty it must be homeomorphic to $\mb C \sm \set{0}$; thus:
\begin{corollary}
	\label{cor:rank_jump}
	If $\rk({\mf g}) - \rk(\mf h') = 1$, then the fundamental group of~\eqref{eq:true_complement} is infinite cyclic.
\end{corollary}
This corresponds to the pure braid group of type $A_1$, i.e.,
the pure braid group on two strands.
(Or to the pure braid group of type $BC$ on one strand,
via the identification $A_1 \simeq B_1$.)

An easy extension then yields the following:

\begin{corollary}
	Suppose that $\rk(\mf g) = 2$.
	Then $\Gamma_Q$ is isomorphic to $\mb Z$, or to $\mb Z^2$, or to the pure $\mf g$-braid group.
\end{corollary}

\begin{proof}
	The possible filtrations of Levi subsystems are listed as
	\begin{equation*}
		\vn \sse \Phi_{\mf g}, \quad \Phi_{\mf h} \sse \Phi_{\mf g}, \quad \vn \sse \Phi_{\mf h} \sse \Phi_{\mf g},
	\end{equation*}
	with $\rk(\Phi_{\mf h}) = 1$.
	The first (generic) one leads to the pure $\mf g$-braid group, and the other two are controlled by Cor.~\ref{cor:rank_jump}.
\end{proof}

\section{Type A}
\label{sec:type_A}

Starting from the section, we will explicitly describe the pure local WMCGs of classical type.

\vspace{5pt}

Let $n \geq 1$ be an integer, and $\mf g = \mf{sl}_{n+1}(\mb C)$.
The standard Cartan subalgebra of traceless diagonal matrices is naturally a subspace of $V \ceqq \mb C^{n+1}$.
We will use the shorthand notation $\ul k \ceqq \set{1,\dc,k}$, for an integer $k \geq 1$, in all that follows.

Denote by $e_1,\dc,e_{n+1} \in V$ the vectors of the canonical basis, and let $\alpha_i = e_i^{\dual} \in V^{\dual}$ be the associated covectors in the dual basis (= the canonical coordinates).
Then we write
\begin{equation*}
	\alpha^-_{ij} \ceqq \alpha_i - \alpha_j, \qquad i \neq j \in \ul{n+1},
\end{equation*}
so that the root system is
\begin{equation*}
	A_n = \Set{ \pm \alpha^-_{ij} | i < j \in \ul{n+1}} \sse \mf t^{\dual},
\end{equation*}
with standard base $\Delta_{\mf g} = \set{ \theta_1,\dc,\theta_n }$, setting $\theta_i \ceqq \alpha^-_{i,i+1}$~\cite[Ch.~VI, \S~4.7]{bourbaki_1968_groupes_et_algebres_de_lie_chapitres_iv_v_vi}.

\subsection{Dynkin diagrams}

Choose an element $A \in \mf t$,
and let $\Phi_{\mf h} \sse A_n$ be the Levi subsystem of its centraliser $\mf h \sse \mf g$.
Reasoning as in \S~\ref{sec:fission}, we can assume that it has base $\Delta_{\mf h} = \Set{ \theta \in \Delta_{\mf g} | \Braket{\theta | A } = 0 }$, which yields a full subdiagram $\mc D_{\mf h}$ of
\begin{equation*}
	\mc D_{A_n} = \dynkin[labels={\theta_1,\theta_2,\theta_{n-1},\theta_n},scale=2] A{}.
\end{equation*}
Concretely,
one has a disjoint union
\begin{equation*}
	\mc D_{\mf h} = \coprod_i \mc D_i \sse \mc D_{A_n},
\end{equation*}
of connected components.
All such components are of type $A$: a component on $k \geq 1$ nodes corresponds to an irreducible root subsystem $A_k \sse \Phi_{\mf h}$, whose simple roots form an unbroken string of length $k$ inside $\Delta_{\mf g}$.

It will now be convenient to break the natural symmetry of the Dynkin diagram,
orienting it left-to-right.
This is topologically w.l.o.g.,
because the diagram-flip automorphism of $\mc D_{A_n}$ $(n \geq 2)$ corresponds to an outer automorphism of the root system $A_n$:
this extends to an outer automorphism of the root system,
which induces a homeomorphism of admissible deformation spaces---%
matching up root-valuations tuples of irregular types.

Consider then the subset
\begin{equation*}
	J' \ceqq \Set{ i \in \ul n | \text{ there is an unbroken string of maximal length starting at } \theta_i }.\fn{
		If such an unbroken string exists, then it is unique.}
\end{equation*}
For any $i \in J'$ we thus find an irreducible component $A_{I'_i} \sse \Phi_{\mf h}$ of rank $\abs{I_i'} > 0$.
This results in a partition $\Delta_{\mf h} = \coprod_{i \in J'} I'_i$, and it will be helpful to introduce the following terminology:

\begin{definition}
	If $S$ and $J$ are finite sets, a $J$-\emph{partition} of $S$ is a surjection $\phi \cl S \thra J$.
	(This is the same as giving a partition $S = \coprod_J I_j$ indexed by $J$, with nonempty parts $I_j \ceqq \phi^{-1}(j) \sse S$.)
\end{definition}

We thus have a $J'$-partition $\Delta_{\mf h} \thra J'$.
This does not quite control the complement~\eqref{eq:true_complement}, since one must also take into account the roots of $A_n$ that have been left out, as we now set out to do.
(The discussion can be simplified in type $A$, but we keep it going in order to streamline the exposition through all classical types.)

\subsection{Kernels}
\label{sec:type_A_kernels}

The space~\eqref{eq:true_complement} is controlled by a partition ``extending'' $\Delta_{\mf h} \thra J'$.
Precisely,
for $i \in \ul{n+1}$ define the subset
\begin{equation}
	\label{eq:second_partition}
	I_i = I_i^{\mf h} \ceqq \set{i} \cup \Set{ j \in \ul{n+1} | \pm \alpha^-_{ij} \in \Phi_{\mf h} } \sse \ul{n+1}.
\end{equation}

\begin{lemma}
	\label{lem:second_partition}
	The subsets~\eqref{eq:second_partition} provide a $J$-partition $\ul{n+1} \thra J$, and the set of parts $J$ has cardinality
	\begin{equation*}
		\abs J = n+1 - \rk(\mf{h'}).
	\end{equation*}
\end{lemma}

\begin{proof}
	First we must show that $I_i \cap I_j \neq \vn$ implies $I_i = I_j$, for $i,j \in \ul{n+1}$.
	This is because the nontrivial root reflections act by
	\begin{equation*}
		\sigma^-_{ij}(\alpha^-_{jk}) = \alpha^-_{ik}, \qquad \sigma^-_{ij} = \sigma_{\alpha^-_{ij}},
	\end{equation*}
	for distinct indices $i,j,k \in \ul{n+1}$, and by hypothesis $\Phi_{\mf h}$ is a root subsystem.

	As for the cardinality of $J$, if $\Phi_{\mf h} = \vn$ the identity is clear; and adding an irreducible component $A_k \sse \Phi_{\mf h}$ reduces it exactly by $k$.
\end{proof}

By construction,
the irreducible component $A_{I'_i} \sse \Phi_{\mf h}$ consists of the set of roots
\begin{equation*}
	\Set{\pm \alpha^-_{jl} | i \leq j < l \leq i + \abs{I_i'} } \sse A_n,
\end{equation*}
for $i \in J'$, and in turn $I_i = \set{i,\dc,i + \abs{I'_i}}$.
Then there is a natural injection $J'\hra J$ which induces a one-to-one correspondence between the sets $I'_i$ and $I_i$.

We can thus also denote by $A_{I_i} \sse \Phi_{\mf h}$ the irreducible components, but now the notation makes sense in general: if $I_i = \set{i}$ then $A_{I_i} \simeq A_0$ stands for the trivial (nonspanning) ``rank-zero'' root systems $A_0 = \vn \sse \mb Ce_i$.
This simply means that $\mb Ce_i \sse \Ker(\Phi_{\mf h})$ (see below).\footnote{
	\label{fn:ordering}
	Note that $J$ is naturally a subset of $\ul{n+1}$ by mapping $I_j \mt \min(I_j)$.
	In particular it inherits a total order, intrinsically coming from the ordering of the eigenvalues of a (traceless) diagonal matrix.}

Now we describe $\Ker(\Phi_{\mf h})$ in two steps: first we consider $\Phi_{\mf h}$ as a Levi subsystem of $\Phi_{\mf{gl}_{n+1}(\mb C)} \sse V^{\dual}$, and then we restrict to the trace-free case.
Concretely, denote by $\wt{\Ker}(\Phi_{\mf h}) \sse V$ the kernel in the general-linear case, so that $\Ker(\Phi_{\mf h}) = \wt{\Ker}(\Phi_{\mf h}) \cap \mf t$ is what we are after---in the special-linear case.

\begin{proposition}
	\label{prop:extended_kernel}
	There is a $\mb C$-linear isomorphism $\mb C^J \xra{\simeq} \wt{\Ker}(\Phi_{\mf h})$, given by mapping
	\begin{equation*}
		\ol e_I \lmt e_I \ceqq \sum_{i \in I} e_i \in V,
	\end{equation*}
	for all $I \in J$, where $\ol e_I$ is a vector of the canonical basis of $\mb C^J$.
\end{proposition}

\begin{proof}
	By Lem.~\ref{lem:second_partition} there is a splitting
	\begin{equation*}
		V \simeq \bops_{I \in J} \mb C^I, \qquad \mb C^I = \bops_{i \in I} \mb Ce_i,
	\end{equation*}
	and the kernel decomposes accordingly.

	In turn each component contains a full copy of the type-$A$ root system, hence the kernel there is spanned by the line through the vector $e_I \in V$ of the statement.
\end{proof}

In brief the vectors of the canonical basis corresponding to each part $I \in J$ are fused in the kernel.
In the end:
\begin{equation}
	\label{eq:kernel_type_A}
	\Ker(\Phi_{\mf h}) = \Biggl\{ \, \sum_{I \in J} \lambda_I e_I \in V \, \Big| \, \sum_J \lambda_I = 0 \, \Biggr\} \sse V.
\end{equation}
We compute
\begin{equation*}
	\dim_{\mb C} \bigl( \Ker(\Phi_{\mf h}) \bigr) = \abs J - 1 = n - \rk(\mf h'),
\end{equation*}
using Lem.~\ref{lem:second_partition}, in accordance with~\eqref{eq:dimension_kernel}.

\subsection{Restricted subsystem and fundamental group}

Write $U \ceqq \Ker(\Phi_{\mf h})$.
To conclude we must remove from it the root-hyperplanes corresponding to (positive) roots $\alpha \in A_n \sm \Phi_{\mf h}$, and it turns out this still yields a root system of type $A$:

\begin{theorem}
	\label{thm:restricted_arrangement_type_A}

	Under the isomorphism of Prop.~\ref{prop:extended_kernel} there is an identification of root systems
	\begin{equation*}
		A_d \simeq \eval[1]{A_n}_U \sse U^{\dual},
	\end{equation*}
	where $d = \dim_{\mb C}(U)$,
	in the notation of~\eqref{eq:restricted_root_subsystem}.
	In particular,
	the hyperplane arrangement in the kernel is of type $A_d$.
\end{theorem}

\begin{proof}
	Introduce the dual basis $\ol \alpha_I = \ol e_I^{\dual}$ of $(\mb C^J)^{\dual}$.
	By construction $\alpha^-_{ij} \in A_n \sm \Phi_{\mf h}$ if and only $I_i \neq I_j$, and restricting such covectors to~\eqref{eq:kernel_type_A} yields all linear functionals of the form $\ol \alpha_I - \ol \alpha_{\wh I} \in U^{\dual}$, with $I \neq \wh I \in J$.
\end{proof}

It follows that $\pi_1\bigl( \bm B(\Phi_{\mf h},\Phi_{\mf g}),A \bigr) \simeq \PB_{d+1}$ in this case.

\begin{remark}
	Note that $d + 1 = \abs J$,
	and this is consistent.
	E.g.,
	if $\Phi_{\mf h} = \vn$ (whence $A$ is a regular vector) the corresponding $J$-partition is discrete,
	so that $d + 1 = n + 1$ and one finds as expected the pure type-$A$ braid group on $n+1$ strands.
	At the opposite end,
	if $\Phi_{\mf h} = \Phi_{\mf g}$ (whence $A = 0$) the corresponding $J$-partition is indiscrete,
	and one finds the trivial group $\PB_1$.

	Moreover,
	when $n \geq 2$,
	recall that we effectively chose a representative subset $\Delta_{\mf h} \sse \Delta_{\mf g}$ within any orbit for the 2-element group of Dynkin-diagram automorphisms:
	but the integer $d \geq 0$ only depends on the orbit itself.
\end{remark}

\subsection{Fission trees}

Finally we can reason recursively: consider a nested Levi subsystem $\Phi_{\wt{\mf h}} \sse \Phi_{\mf h}$.
Splitting into irreducible components, it follows that	the partition $\ul{n+1} \thra \wt J$ (associated with $\Phi_{\wt{\mf h}}$ as in~\eqref{eq:second_partition}) is a refinement of the $J$-partition associated with $\Phi_{\mf h}$.
Precisely,
for any $i \in \wt J$ there exists $\phi(i) \in J$ such that $I^{\wt{\mf h}}_i \sse I^{\mf h}_{\phi(i)}$, i.e.,
there is a (new) $J$-partition $\phi \cl \wt J \thra J$.

Hence a filtration
\begin{equation*}
	\Phi_{\mf h_1} \sse \dm \sse \Phi_{\mf h_p} \sse \Phi_{\mf h_{p+1}} \ceqq A_n,
\end{equation*}
of Levi subsystems,
corresponds to a decreasing sequence of sets
\begin{equation}
	\label{eq:fission_tree_sequence}
	J_1 \lxra{\phi_1} J_2 \lra \dm \lra J_p \lxra{\phi_p} J_{p+1} \ceqq \set{\ast},
\end{equation}
where $J_l$ is the set of parts corresponding to $\Phi_{\mf h_l}$, as in Lem.~\ref{lem:second_partition}---for $l \in \ul p$.
This is the same as considering the disjoint union $\mc T_0 \ceqq \coprod_{l = 1}^{p+1} J_l$, and giving a single function
\begin{equation*}
	\bm \phi \cl \mc T_0 \sm \set{\ast} \lra \mc T_0, \qquad J_l \ni i \lmt \phi_l(i) \in J_{l+1}.
\end{equation*}

\begin{definition}[Fission tree]
	\label{def:fission_tree}
	The \emph{fission tree} $\mc T_Q$ of~\eqref{eq:fission_tree_sequence} is the tree with nodes $\mc T_0$, such that $\bm \phi(i) \in \mc T_0$ is the parent-node of $i \in \mc T_0 \sm \set{\ast}$.
\end{definition}

Hence $J_1 \sse \mc T_0$ are the leaves and $\ast \in J_{p+1}$ is the root, while $\abs{J_l}$ is the number of nodes at level $l \in \ul{p+1}$;
note that by construction one has $\abs{J_1} \leq n+1$. (The equality corresponds to $H_1 = T$, in the notation of \S~\ref{sec:fission}.)
Set finally
\begin{equation*}
	k_i \ceqq \abs{ \bm \phi^{-1}(i)} \in \mb Z_{\geq 0}, \qquad i \in \mc T_0,
\end{equation*}
which is the number of child-nodes of $i \in \mc T_0$.\fn{
	All sets $\bm \phi^{-1}(i) \sse \mc T_0$ come with a total order: cf.~Fn.~\ref{fn:ordering}.}

\begin{theorem}
	\label{thm:tree_gives_wmcg_type_A}

	There is a group isomorphism $\Gamma_Q \simeq \prod_{\mc T_0} \PB_{k_i}$.
	Conversely,
	pure local WMCGs of type $A$ exhaust finite products of pure braid groups---%
	of type $A$.
\end{theorem}

\begin{proof}
	By construction a node $i \in J_l$ corresponds to an irreducible component of $\Phi_{\mf h_l}$, which splits into $k_i \geq 0$ irreducible components inside $\Phi_{\mf h_{l-1}} \sse \Phi_{\mf h_l}$, corresponding to the child-nodes $j \in \bm \phi^{-1}(i) \sse J_{l-1}$.
	By Thm.~\ref{thm:restricted_arrangement_type_A} this yields a pure braid group on $k_i$ strands, sitting inside $\Gamma_Q$.
	This gives independent factors at each level of the tree, and the conclusion follows from the splitting~\eqref{eq:pure_local_WMCG_factors}---taking the product over all levels.

	For the second statement, given any sequence of integers $n_i \geq 0$ with finite support we can construct a fission tree having precisely $n_i$ nodes with $i \geq 1$ child-nodes (in many ways, cf.~Exmp.~\ref{ex:examples_presentations}).
	In that case
	\begin{equation}
		\label{eq:cabled_braid_group_type_A_distinct_factors}
		\Gamma_Q \simeq \prod_{i \geq 1} \PB^{n_i}_i. \qedhere
	\end{equation}
\end{proof}

\begin{remark}[Low-order and irreducible presentations]
	Thm.~\ref{thm:tree_gives_wmcg_type_A} implies that different trees can lead to isomorphic groups, e.g.,
	a ``presentation'' by a tree of minimal height is obtained by splitting nodes as quick as possible.

	Conversely we can consider a tree with a single node splitting at each level: all fission subsystems are then \emph{irreducible}, so there are integers $n_1 \leq \dm \leq n_p \leq n$ such that~\eqref{eq:root_filtration} becomes
	\begin{equation*}
		A_{n_1} \sse \dm \sse A_{n_p} \sse A_n,
	\end{equation*}
	with embedding on the first slots.
\end{remark}

\begin{example}[Examples of presentations]
	\label{ex:examples_presentations}
	Given the semisimple rank $n = 8$,
	consider the irregular type $Q^I = T_1 x + T_2 x^2 + T_3 x^3$, taking the following coordinate vectors $\bm \alpha_i = \bigr(\alpha_1(T_i),\dc,\alpha_9(T_i) \bigr) \in \mb C^9$:
	\begin{equation*}
		\bm \alpha_1 = (4,3,2,1,0,-1,-2,-3,-4), \quad \bm \alpha_2 = (4,4,3,2,1,0,-3,-4,-7),
	\end{equation*}
	and
	\begin{equation*}
		\bm \alpha_3 = (2,2,1,1,1,0,0,0,-7).
	\end{equation*}
	Then the sequence of Levi subsystems is
	\begin{equation*}
		\vn \sse A_1 \sse A_1 \ops A_2 \ops A_2 \sse A_8,
	\end{equation*}
	and the associated pure local WMCG is $\Gamma_{Q^I} \simeq \PB_2 \times \PB_3^2 \times \PB_4$.
	The corresponding fission tree is in Fig.~\ref{fig:exmp_tree_1}.

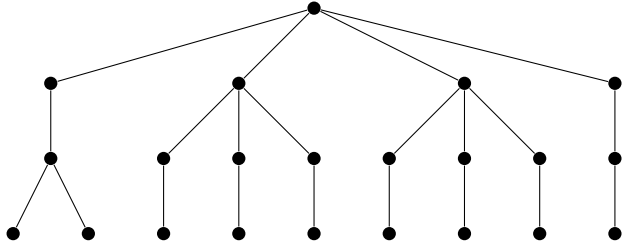
	\begin{figure}
		\begin{center}
			\begin{tikzpicture}
				\foreach \name/\x/\y in {A1/1/0,A2/2/0,A3/3/0,A4/4/0,A5/5/0,A6/6/0,A7/7/0,A8/8/0,A9/9/0,
						C1/1.5/2,C2/4/2,C3/7/2,C4/9/2,
						B1/1.5/1,B2/3/1,B3/4/1,B4/5/1,B5/6/1,B6/7/1,B7/8/1,B8/9/1,
						D1/5/3}
				\vertex (\name) at (\x,\y){};
				\foreach \from/\to in {A1/B1,A2/B1,A3/B2,A4/B3,A5/B4,A6/B5,A7/B6,A8/B7,A9/B8,
						B1/C1,B2/C2,B3/C2,B4/C2,B5/C3,B6/C3,B7/C3,B8/C4,
						D1/C1,D1/C2,D1/C3,D1/C4}
				\draw (\from) -- (\to);
			\end{tikzpicture} \quad .
		\end{center}
		\caption{\label{fig:exmp_tree_1}
			Example of fission tree (I)}
	\end{figure}

	But this is also the fundamental group of the space of admissible deformations of $Q^{II} = T_1 x + T_2 x^2$, taking coordinate vectors
	\begin{equation*}
		\bm \alpha_1 = (4,3,2,1,0,-1,-2,-3,-4), \quad \bm \alpha_2 = (4,1,1,0,0,0,-2,-2,-2),
	\end{equation*}
	which yields the low-order presentation, and has associated filtration
	\begin{equation*}
		\vn \sse A_1 \ops A_2 \ops A_2 \sse A_8.
	\end{equation*}
	The (minimal-height) fission tree is then as in Fig.~\ref{fig:example_fission_tree_2}.

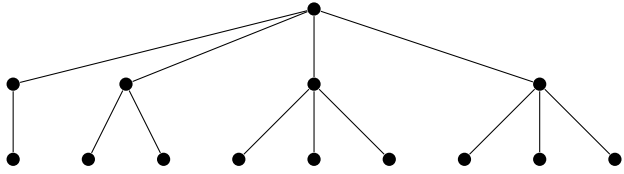
	\begin{figure}
		\begin{center}
			\begin{tikzpicture}
				\foreach \name/\x/\y in {A1/1/0,A2/2/0,A3/3/0,A4/4/0,A5/5/0,A6/6/0,A7/7/0,A8/8/0,A9/9/0,
						B1/1/1,B2/2.5/1,B3/5/1,B4/8/1,
						C1/5/2}
				\vertex (\name) at (\x,\y){};
				\foreach \from/\to in {A1/B1,A2/B2,A3/B2,A4/B3,A5/B3,A6/B3,A7/B4,A8/B4,A9/B4,
						C1/B1,C1/B2,C1/B3,C1/B4}
				\draw (\from) -- (\to);
			\end{tikzpicture} \quad .
		\end{center}
		\caption{\label{fig:example_fission_tree_2}
			Example of fission tree (II)}
	\end{figure}

	Finally, this pure local WMCG is also the fundamental group of the space of admissible deformations of $Q^{III} = T_1 x + T_2 x^2 + T_3 x^3 + T_4 x^4$, taking coordinate vectors
	\begin{equation*}
		\bm \alpha_1 = (4,3,2,1,0,-1,-2,-3,-4), \quad \bm \alpha_2 = (4,4,3,2,1,0,-3,-4,-7),
	\end{equation*}
	and
	\begin{equation*}
		\bm \alpha_3 = (2,2,2,2,1,0,-3,-3,-3), \quad \bm \alpha_4 = (1,1,1,1,1,1,0,-2,-4).
	\end{equation*}
	This yields the irreducible presentation, with filtration
	\begin{equation*}
		\vn \sse A_1 \sse A_3 \sse A_5 \sse A_8,
	\end{equation*}
	and the fission tree is as in Fig.~\ref{fig:exmp_fission_tree_3}. \qedhere

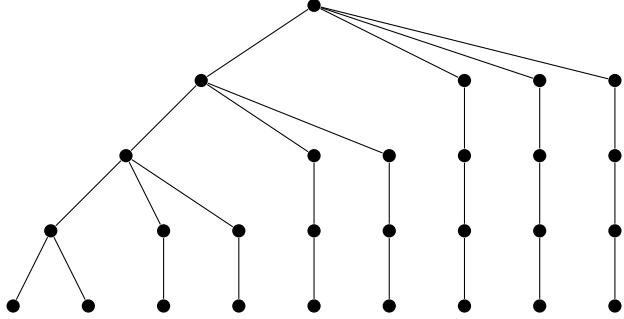
	\begin{figure}
		\begin{center}
			\begin{tikzpicture}
				\foreach \name/\x/\y in {A1/1/0,A2/2/0,A3/3/0,A4/4/0,A5/5/0,A6/6/0,A7/7/0,A8/8/0,A9/9/0,
						B1/1.5/1,B2/3/1,B3/4/1,B4/5/1,B5/6/1,B6/7/1,B7/8/1,B8/9/1,
						C1/2.5/2,C2/5/2,C3/6/2,C4/7/2,C5/8/2,C6/9/2,
						D1/3.5/3,D2/7/3,D3/8/3,D4/9/3,
						E1/5/4}
				\vertex (\name) at (\x,\y){};
				\foreach \from/\to in {A1/B1,A2/B1,A3/B2,A4/B3,A5/B4,A6/B5,A7/B6,A8/B7,A9/B8,
						B1/C1,B2/C1,B3/C1,B4/C2,B5/C3,B6/C4,B7/C5,B8/C6,
						C1/D1,C2/D1,C3/D1,C4/D2,C5/D3,C6/D4,
						D1/E1,D2/E1,D3/E1,D4/E1}
				\draw (\from) -- (\to);
			\end{tikzpicture} \quad .
		\end{center}
		\caption{
			\label{fig:exmp_fission_tree_3}
			Example of fission tree (III)}
	\end{figure}
\end{example}

\section{Type B/C}
\label{sec:type_BC}

Let now $n \geq 2$ be an integer, and $\mf g = \mf{so}_{2n+1}(\mb C)$.
The standard Cartan subalgebra $\mf t \sse \mf g$ is identified with $\mb C^n = \bops_i \mb Ce_i$, and we retain the notations of \S~\ref{sec:type_A}.

The usual choice of base for the root system is $\Delta_{\mf g} = \set{\theta_1,\dc,\theta_n}$ with
\begin{equation*}
	\theta_i = \alpha^-_{i,i+1}, \qquad i \in \ul{n-1},
\end{equation*}
as for $A_{n-1}$, plus the \emph{short} root $\theta_n \ceqq \alpha_n \in \mf t^{\dual}$.
We then have
\begin{equation*}
	B_n = \Set{ \pm \alpha^-_{ij}, \, \pm \alpha^+_{ij} | i < j \in \ul n } \cup \Set{ \pm \alpha_i | i \in \ul n } \sse \mf t^{\dual},
\end{equation*}
writing $\alpha^+_{ij} \ceqq \alpha_i + \alpha_j$~\cite[Ch.~VI, \S~4.5]{bourbaki_1968_groupes_et_algebres_de_lie_chapitres_iv_v_vi}.
(We tacitly identify $A_{n-1}$ with the Levi subsystem of $B_n$ corresponding to the subset of simple roots $\set{\theta_1,\dc,\theta_{n-1}}$.)

\subsection{Dynkin diagrams}

Let as above $\Phi_{\mf h} \sse B_n$ be the Levi subsystem associated to an element $A \in \mf t$, and suppose that $\Delta_{\mf h} = \Delta_{\mf g} \cap \set{A}^{\perp}$ lies in $\Delta_{\mf g}$.
We must now consider whether the Dynkin (sub)diagram $\mc D_{\mf h}$ contains the rightmost node of
\begin{equation*}
	\mc D_{B_n} = \dynkin[labels={\theta_1,\theta_2,,\theta_{n-1},\theta_n},scale=2] B{}.
\end{equation*}

If it does not, then $\Delta_{\mf h}$ only contains long roots, whence $\Phi_{\mf h}$ is of type $A$:
precisely,
it is a Levi subsystem of $A_{n-1}$,
and it splits into irreducible type-$A$ components---%
as in \S~\ref{sec:type_A}.
Otherwise,
there exists a unique positive integer $m \leq n$ such that one has a disjoint union
\begin{equation}
	\label{eq:base_split}
	\Delta_{\mf h} = \Delta^A_{\mf h} \cup \Delta^B_{\mf h}, \qquad \Delta_{\mf h}^B \ceqq \set{\theta_{n-m+1},\dc,\theta_n} \sse \Delta_{\mf g},
\end{equation}
for a suitable subset $\Delta_{\mf h}^A \sse \set{\theta_1,\dc,\theta_{n-m-1}}$.
(Whence $\theta_{n-m} \not\in \Delta_{\mf h}$.)

In the latter case,
one finds a direct splitting $\Phi_{\mf h} \simeq \Phi_{\mf h}^A \ops B_m$, where in turn:
(i) one has $\Phi_{\mf h}^A \sse A_{n-m-1}$,
sitting inside the subspace $\bops_{i = 1}^{n-m} \mb C \cdot \alpha_i$ (or the trace-free version);
and (ii) one has $B_m \sse \bops_{i = n-m+1}^n \mb C \cdot \alpha_i$.
Conversely,
the former case can be seen as a particular case by extending this notation down to $m = 0$ in the natural way (so that $B_0 = \vn \sse (0) = \mb C^0$,
etc.).

\begin{remark}
	\label{rmk:B_component}

	One should view $B_m \sse \Phi_{\mf h}$ as a genuine type-$B$ component even when $m = 1$:
	the abstract exceptional isomorphism $B_1 \simeq A_1$ plays no role,
	as the two types are distinguished---%
	as root subsystems of $B_n$---%
	by looking at root lengths.

	Moreover,
	contrary to type $A$,
	here there are \emph{no} nontrivial Dynkin-diagram automorphisms,
	and in particular the notion of ``rightmost'' node is canonical.
\end{remark}

\subsection{Kernels}

Let again $U = \Ker(\Phi_{\mf h}) \sse \mf t$.
By construction
\begin{equation*}
	U = \Ker(\Phi^A_{\mf h}) \cap \Ker(B_m),
\end{equation*}
for a suitable integer $m \in \set{0,\dc,n}$.

With the above notation one has $\Ker(B_m) = \mb C^{n-m} \times (0) \sse \mf t$, and we conclude by Prop.~\ref{prop:extended_kernel}: there is a linear isomorphism
\begin{equation}
	\label{eq:kernel_type_B}
	\mb C^J \simeq U \sse \mb C^{n-m} \times (0),
\end{equation}
where $J$ is the index set of the partition of $\ul{n-m}$ associated with $\Phi_{\mf h}^A$, as in~\eqref{eq:second_partition}.\fn{Cf.~\cite{doucot_rembado_yamakawa_twisted_local_g_wild_mapping_class_groups} for a relation with \emph{signed} set-partitions,
	going back (at least) to~\cite{reiner_1997_non_crossing_partitions_for_classical_reflection_groups}.}
(In this case there is no further restriction to the trace-free subspace.)

\subsection{Restricted arrangement and fundamental group}

Finally we will describe the hyperplane arrangement inside the kernel, provided by the kernel of the (positive) roots $\alpha \in B_n \sm \Phi_{\mf h}$ after restriction to $U$.

\begin{theorem}
	\label{thm:restricted_arrangement_type_B}

	The hyperplane arrangement in the kernel is of type $B_d/C_d$, where $d = \dim_{\mb C}(U)$.
	Moreover, if no component of the $J$-partition is a singleton, then under the isomorphism~\eqref{eq:kernel_type_A} there is an identification of root systems
	\begin{equation}
		\label{eq:restricted_roots_type_B}
		BC_d \simeq \eval[1]{B_n}_U \sse U^{\dual},
	\end{equation}
	in the notation of~\eqref{eq:restricted_root_subsystem}.
\end{theorem}

\begin{proof}
	Computing the restrictions of all roots in the complement of $\Phi_{\mf h}$ shows that one always has the inclusion
	\begin{equation*}
		\Set{ \pm (\ol \alpha_I - \ol \alpha_{\wh I}), \, \pm (\ol \alpha_I + \ol \alpha_{\wh I}) | I \neq \wh I \in J } \cup \set{ \pm \ol \alpha_I | I \in J } \sse \eval[1]{B_n}_U,
	\end{equation*}
	in the notation in the proof of Thm.~\ref{thm:restricted_arrangement_type_A}.
	This is the full root system of type $B_d$,
	and furthermore the covectors $\pm 2 \ol \alpha_I \in U^{\dual}$ appear from the restrictions of $\pm \alpha^+_{ij} \in B_n$ if and only if there exists a pair $(i,j)$ with $i \in J$ and $j \in I_i \sm \set{i}$.

	This exhausts the covectors obtained upon restriction to $U$, and so the hyperplane arrangement is of type $B/C$.
\end{proof}

\begin{remark}
	The hyperplanes arrangements are always those of a root system, so their reflection groups are crystallographic; but the set of restricted functional themselves are \emph{not} root systems in general.
\end{remark}

It follows that $\pi_1 \bigl( \bm B(\Phi_{\mf h},\Phi_{\mf g}),A \bigr) \simeq \PB^{BC}_d$ in this case.
In particular,
just as in type $A$,
the number of strands is the same as the number of components of the set partition determined by the type-$A$ factors of $\Phi_{\mf h}$.

\subsection{Bichromatic fission trees}

Again we can now reason recursively.
In brief, a component of type $B$ will produce a pure braid group of type $B/C$ upon breaking to a Levi subsystem, and this will continue until such subsystems still contain a type-$B$ component.
At some point this might stop, in which case we will get back to only having (sub)components of type $A$---and pure braid groups of type $A$.

This leads to the following natural generalisation of Def.~\ref{def:fission_tree}.
Denote by $\set{g,b}$ the set of colours ``green'' and ``blue'', ordered by $g \leq b$; then:
\begin{definition}[Bichromatic fission tree]
	\label{def:bichromatic_fission_tree}

	A \emph{bichromatic fission tree} is a fission tree $\mc T = \mc T_Q$ equipped with a \emph{colour function} $c \cl \mc T_0 \to \set{g,b}$;
	in turn a colour function satisfies:
	\begin{enumerate}
		\item $c \bigl(\bm \phi(i)\bigr) \geq c(i)$ for $i \in \mc T_0 \sm \set{\ast}$;

		\item $\abs{\bm \phi^{-1}(i) \cap c^{-1}(b)} \leq 1$ for $i \in \mc T_0$.
	\end{enumerate}
\end{definition}

The conditions mean that green nodes have green child-nodes, and that any node has at most one blue child-node, respectively.
Now consider a fission sequence
\begin{equation*}
	\Phi_{\mf h_1} \sse \dm \sse \Phi_{\mf h_p} \sse  \Phi_{\mf h_{p+1}} = B_n.
\end{equation*}
As above this leads to a type-$A$ filtration
\begin{equation*}
	\Phi_{\mf h_1}^A \sse \dm \Phi_{\mf h_p}^A \sse A_{n-1},
\end{equation*}
and to a filtration by irreducible subsystems
\begin{equation*}
	B_{m_1} \sse \dm \sse B_{m_p} \sse B_n,
\end{equation*}
with embeddings on the last slots (at each step).

The algorithm to assign a bichromatic fission tree to such double filtrations is the following.
A node $i \in J_l$ corresponds to an irreducible component of the subsystem $\Phi_{\mf h_l} \sse B_n$: put a green node for each type-$A$ component (including ``trivial'' ones,
cf.~\S~\ref{sec:type_A_kernels}), and a blue node if $m_l > 0$; then define $j = \bm{\phi(i)} \in J_{l+1}$ if the irreducible component of $\Phi_{\mf h_{l+1}}$ associated with $j$ contains the irreducible component of $\Phi_{\mf h_l} \sse \Phi_{\mf h_{l+1}}$ associated with $i \in J_l$.
Finally,
we extend the notation of \S~\ref{sec:type_A} by redefining
\begin{equation*}
	k_i \ceqq \abs{\bm \phi^{-1}(i) \cap c^{-1}(g)} \geq 0, \qquad i \in \mc T_0,
\end{equation*}
which is the number of \emph{green} child-nodes of $i$.

\begin{theorem}
	\label{thm:tree_gives_wmcg_type_BC}
	There is a group isomorphism
	\begin{equation}
		\label{eq:type_BC_cabled_braid_group}
		\Gamma_Q \simeq \prod_{c^{-1}(g)} (\PB_{k_i}) \times \prod_{c^{-1}(b)} \PB^{BC}_{k_i}.
	\end{equation}
	Conversely, pure local WMCGs of type $B$ exhaust finite products of pure braid groups of types $A$ and $B/C$.
\end{theorem}

\begin{proof}
	For the first statement, the new situation (with respect to Thm.~\ref{thm:tree_gives_wmcg_type_A}) is that a blue node $i \in \mc T_0$ yields a pure braid group of type $B_k/C_k$, where $k$ is the number of its green child-nodes---corresponding to the decomposition of $B_{m_l} \sse \Phi_{\mf h_{l+1}}$ into type-$A$ irreducible components for $\Phi_{\mf h_l} \sse \Phi_{\mf h_{l+1}}$.

	For the second statement consider trees where no green node splits.
	If there are $n_i \geq 1$ blue nodes with $i \geq 1$ green child-nodes this yields
	\begin{equation*}
		\Gamma_Q \simeq \prod_{i \geq 1} \bigl(\PB^{BC}_i\bigr)^{n_i},
	\end{equation*}
	which is an arbitrary finite product (analogously to~\eqref{eq:cabled_braid_group_type_A_distinct_factors}), and type-$A$ factors are then obtained by splaying some green node.
\end{proof}

\subsection{Type C}

Taking $n \geq 3$ and $\mf g = \mf{sp}_{2n}(\mb C)$, with root system
\begin{equation*}
	C_n = \Set{ \pm \alpha^-_{ij}, \, \pm \alpha^+_{ij} | i < j \in \ul n } \cup \Set{ \pm 2\alpha_i | i \in \ul n } \sse \mf t^{\dual},
\end{equation*}
yields essentially the same situation (see~\cite[Ch.~VI, \S~4.6]{bourbaki_1968_groupes_et_algebres_de_lie_chapitres_iv_v_vi} for the construction of the root system).
In brief this is because $C_n$ is the dual/inverse root system of $B_n$, and the Dynkin diagram has an analogous shape:
\begin{equation*}
	\mc D_{C_n} = \dynkin[labels={\theta_1,\theta_2,,\theta_{n-1},\theta_n},scale=2] C{},
\end{equation*}
where now $\theta_n \ceqq 2\alpha_n$ is the \emph{long} simple root.

This leads to the same hyperplane arrangements, and slight variations of the above arguments yield proofs of theorems analogous to Thm.~\ref{thm:restricted_arrangement_type_B}--\ref{thm:tree_gives_wmcg_type_BC}:
keeping the same notation,
the only difference is that $\eval[1]{C_n}_U \sse U^{\dual}$ always contains the roots of type $C_d$,
and that the roots $\pm \ol\alpha_I$ appear if $m > 0$ (upon restricting $\pm \alpha^+_{in} \in C_n$,
for $i \in \ul{n-1}$).
In particular,
bichromatic fission trees control pure local WMCGs of types $A$, $B$, and $C$.

\begin{remark}
	Note that the partition introduced in Lem.~\ref{lem:second_partition}, in type $A$, classifies \emph{all} root subsystems.
	In particular,
	as it is well-known,
	they are all Levi---%
	and obtained from fission.

	This is false in types $B/C$: suffices to take a root subsystem which contains more than one component of type $B/C$.
	Nonetheless,
	the idea of Lem.~\ref{lem:second_partition} can be extended to treat all the classical types, generalising the partition $\ul{n+1} \thra J$ to other combinatoric objects (which retain more information than the Dynkin diagram, cf.~\cite{rembado_2024_a_colourful_classification_of_quasi_root_systems_and_hyperplane_arrangements}).
\end{remark}

In the next section we will see that yet another generalisation is necessary for the last classical type, which finally leads to \emph{noncrystallographic} arrangements.

\section{Type D}
\label{sec:type_D}

For an integer $n \geq 4$, let finally $\mf g = \mf{so}_{2n}(\mb C)$.
The standard Cartan subalgebra $\mf t \sse \mf g$ is identified with $\mb C^n = \bops_i \mb Ce_i$, and we retain the notations of \S\S~\ref{sec:type_A}--\ref{sec:type_BC}.

The usual choice of base is $\Delta_{\mf g} = \set{\theta_1,\dc,\theta_n}$ with
\begin{equation*}
	\theta_i = \alpha^-_{i,i+1}, \qquad i \in \ul{n-1},
\end{equation*}
as for $A_{n-1}$, while $\theta_n \ceqq \alpha^+_{n-1,n} \in \mf t^{\dual}$~\cite[Ch.~VI, \S~4.8]{bourbaki_1968_groupes_et_algebres_de_lie_chapitres_iv_v_vi}.
We then have
\begin{equation*}
	D_n = \Set{ \pm \alpha^-_{ij}, \, \pm \alpha^+_{ij} | i < j \in \ul n }.
\end{equation*}
(Again,
we tacitly identify $A_{n-1}$ with the Levi subsystem determined by $\Delta_{\mf g} \sm \set{\theta_n}$.)

\subsection{Dynkin diagrams}

Let $\Phi_{\mf h} \sse D_n$ be a Levi subsystem,
and suppose as usual that it corresponds to a subdiagram $\mc D_{\mf h} \sse \mc D_{D_n}$ of the Dynkin diagram
\begin{equation*}
	\mc D_{D_n} = \dynkin[labels={\theta_1,\theta_2,,,\theta_{n-1},\theta_n},scale=2] D{}.
\end{equation*}

Somewhat analogously to the types $B/C$, the question is whether $\mc D_{\mf h}$ contains the two rightmost nodes or not,
with the following caveat.
If $n \geq 5$,
then the Dynkin diagram can be given a canonical orientation left-to-right,
as the unique nontrivial diagram automorphism swaps $\theta_{n-1}$ and $\theta_n$---%
thereby flipping $\mc D_{D_n}$ upside down.
If $n = 4$,
instead,
the "peripheral" nodes in $\Delta_{\mf g} \sm \set{\theta_2}$ are cyclically permuted under triality:
we then break part of this symmetry by choosing an orientation left-to-right,
similarly to what we did in type $A_n$ ($n \geq 2$),
cf.~Rmkk.~\ref{rmk:about_D_components} and~\ref{rmk:triality}.

Suppose now that $\mc D_{\mf h}$ does \emph{not} contain the two rightmost nodes.
Then $\Delta_{\mf h}$ is contained within one of the two type-$A$ legs:
either $\Delta_{\mf h} \sse \set{\theta_1,\dc,\theta_{n-1}}$,
or $\Delta_{\mf h} \sse \set{\theta_1,\dc,\theta_{n-2},\theta_n}$.
The first subcase yields $\Phi_{\mf h} \sse A_{n-1}$,
and so the Levi subsystem splits into type-$A$ components as usual.
The second subcase is equivalent to the first one,
up to a homeomorphism of admissible deformation spaces;
suffices to consider the aforementioned flip of the fork.

Otherwise,
there exists a unique integer $m \in \set{2,\dc,n}$ such that one has a disjoint union
\begin{equation}
	\Delta_{\mf h} = \Delta_{\mf h}^A \cup \Delta_{\mf h}^D,
	\qquad \Delta_{\mf h}^D \ceqq \set{\theta_{n-m+1},\dc,\theta_n} \sse \Delta_{\mf h},
\end{equation}
for a suitable subset $\Delta_{\mf h}^A \sse \set{\theta_1,\dc,\theta_{n-m-1}}$. (Whence $\theta_{n-m} \not\in\Delta_{\mf h}$,
compare with~\eqref{eq:base_split}.)
Then one has a direct splitting $\Phi_{\mf h} \simeq \Phi_{\mf h}^A \ops D_m$,
with $\Phi_{\mf h}^A \sse A_{n-m-1}$,
as in type $B/C$.

Furthermore,
one can extend this notation to include $m = 0$,
in which case $\Delta_{\mf h}^D = \vn$,
recovering the situation where $\Phi_{\mf h} \sse A_{n-1}$.

\begin{remark}
	\label{rmk:about_D_components}

	Again considering $m = 1$ is redundant,
	up to isomorphisms of pure local WMCGs:
	it corresponds to having $\theta_n \in \Delta_{\mf h} \sse \set{\theta_1,\dc,\theta_{n-2},\theta_n}$,
	which can be flipped inside the upper leg of the Dynkin diagram.

	Moreover,
	analogously to Rmk.~\ref{rmk:B_component},
	one should regard $D_m \sse \Phi_{\mf h}$ as a genuine type-$D$ part even when $m \in \set{2,3}$:
	the abstract exceptional isomorphisms $D_2 \simeq A_1 \ops A_1$ and $D_3 \simeq A_3$ play no role.
	Precisely,
	if $n \geq 5$ then we distinguish low-rank type-$D$ parts by their position in the Dynkin diagram,
	as \emph{no} automorphism of the root system can move the rightmost component $D_m$ to the left.
	Again this is clear for the outer part, consisting of automorphisms of $\mc D_{D_n}$,
	and so we focus on the Weyl-group action on the upper leg.
	Then the point is that a rank-$3$ Levi subsystem of the form
	\begin{equation}
		\Phi_{\mf h} = \Set{ \pm \alpha^-_{ij} | i \neq j \in \set{k,k+1,k+2} } \sse D_n,
		\qquad k \in \set{1,\dc,n-2},
	\end{equation}
	which we consider of "inner" type $A_3$,
	cannot be mapped by the Weyl group onto
	\begin{equation}
		\Phi_{\mf h'} = \Set{ \pm \alpha^-_{ij},\pm \alpha^+_{ij} | i \neq j \in \set{ n-2,n-1,n } },
	\end{equation}
	which we---%
	thus---%
	consider of "inner" type $D_3$.
	A posteriori,
	this follows from Thm.~\ref{thm:restricted_arrangement_type_D},
	as the isomorphism class of the restricted hyperplane arrangement in the kernel is an $\Aut(D_n)$-invariant.
	Namely,
	while the kernels both have complex dimension $n-3$,
	the numbers of hyperplanes do \emph{not} match up:
	there is a difference of $n - 4 \geq 1$,
	which is precisely the number of coordinate hyperplanes of $\mb C^{n-3}$ left out in the case of $\Phi_{\mf h}$.
	Analogously,
	a Levi subsystem of the form
	\begin{equation}
		\Phi_{\mf h} = \Set{ \pm \alpha^-_{i,i+1}, \pm \alpha^-_{j,j+1} | 1 \leq i \leq j - 3 \leq n - 4 } \sse D_n,
	\end{equation}
	which we consider of "inner" type $A_1 \ops A_1$,
	cannot be mapped by the Weyl group onto
	\begin{equation}
		\Phi_{\mf h'} = \Set{ \pm \alpha^-_{n-1,n}, \pm \alpha^+_{n-1,n }},
	\end{equation}
	which we (thus) consider of "inner" type $D_2$.

	Finally,
	if $n = 4$,
	a "part of type $D_2$" should be rather invariantly defined as \emph{any} of the possible embeddings $\mc D_{D_2} \sse \mc D_{D_4}$,
	viz.,
	any of the three triality-symmetric pairs of peripheral nodes (the central node cannot be used, in order to preserve two connected components).~Analogously,
	a "part/component of type $D_3$" is any embedding of the corresponding Dynkin diagram,
	viz.,
	any of the three connected subdiagrams on three nodes (which pass trough the central one;
	cf.~Rmk.~\ref{rmk:triality}).\fn{
		Cf.~\cite{doucot_rembado_yamakawa_twisted_local_g_wild_mapping_class_groups} for an intrinsic definition of the "type-$D$" part of a Levi subsystem of $D_n$---%
		without need to put it in dominant form.}
\end{remark}

\subsection{Kernels}

In the above notation,
for any $m \in \set{0} \cup \set{2,\dc,n}$ one finds
\begin{equation*}
	\Ker(\Phi_{\mf h}) = \Ker(\Phi^A_{\mf h}) \cap \Ker(D_m),
\end{equation*}
and again one has $\Ker(D_m) = \mb C^{n-m} \times (0) \sse \mf t$.
Thus $U = \Ker(\Phi_{\mf h})$ only depends on the type-$A$ irreducible components of $\Phi_{\mf h}$ (cf.~\eqref{eq:kernel_type_B}),
and its dimension is equal to the number of parts of the induced partition $\ul{n-m} \thra J$.\fn{
	This matches up with the notion of \emph{type}-$D_n$ set-partitions,
	which are signed partitions of $\set{\pm 1,\dc,\pm n}$ where the positive part of the zero block---%
	if present---%
	has at least two elements (and cf.,
	again,
	\cite{reiner_1997_non_crossing_partitions_for_classical_reflection_groups}.}

\subsection{Restricted arrangement and fundamental group}

To introduce our new example, choose two integers $r,s \geq 0$.
Consider the following hyperplane arrangement inside $\mb C^{r+s}$: it contains $\Ker(\alpha^{\pm}_{ij})$ for $ i \neq j \in \ul{r+s}$ (i.e.,
the root hyperplanes of type $D_{r+s}$),
and the hyperplanes $\Ker(\alpha_i)$ for $i \in \ul r$.
Hence $\mb C^r \times \set{0} \sse \mb C^{r+s}$ contains the root hyperplanes of type $B_r/C_r$---%
but there is no direct splitting.

We will say that this is an hyperplane arrangement of ``exotic'' type $(B_r/C_r)D_s$.

\begin{remark}
	\label{rmk:noncrystallography}

	The reflection group generated by the latter hyperplane arrangement is the Weyl group of type $B_{r+s}/C_{r+s}$ if $r > 0$, else it is the Weyl group of type $D_s$; this is thus always crystallographic, but the hyperplane arrangement itself is \emph{not} that of a root system in general:
	the simplest such example consists of 7 hyperplanes in $\mb C^3$,
	corresponding to an inclusion $A_1 \sse D_4$---%
	at any node of $\mc D_{D_4}$,
	cf.~Rmk.~\ref{rmk:triality}.
	(There are no irreducible, reduced, rank-3 root systems with 14 roots;
	and one can actually prove that this is not the complement of the reflection arrangement of any \emph{complex} reflection group.)
\end{remark}

\begin{theorem}
	\label{thm:restricted_arrangement_type_D}

	Let $d = \dim_{\mb C}(U)$.
	Then there are two cases:
	\begin{enumerate}
		\item if $m > 0$,
		      then the hyperplane arrangement in the kernel is of type $B_d/C_d$;\fn{Cf.~the ``Case 1'' in the classification at pp.~71--72/80 of~\cite{howlett_1980_normalizers_of_parabolic_subgroups_of_reflection_groups}.
		      }

		\item and if $m = 0$,
		      then the hyperplane arrangement in the kernel is of type $(B_r/C_r)D_s$, where $r \leq d$ is the number of \emph{nontrivial} irreducible components of $\Phi_{\mf h}^A = \Phi_{\mf h}$, and $s = d - r$.
	\end{enumerate}
\end{theorem}

\begin{proof}
	In the notation of the proofs of Thmm.~\ref{thm:restricted_arrangement_type_A} and~\ref{thm:restricted_arrangement_type_B}, one always has
	\begin{equation*}
		\Set{ \pm (\ol \alpha_I - \ol \alpha_{\wh I}), \, \pm (\ol \alpha_I + \ol \alpha_{\wh I}) | I \neq \wh I \in J } \sse \eval[1]{D_n}_U,
	\end{equation*}
	which is the full root system of type $D_d$.
	Moreover,
	some of the covectors $\pm \ol \alpha_I, \pm 2\ol \alpha_I \in U^{\dual}$ may also appear.
	Namely,
	if $I \in J$ is not a singleton then $\pm 2 \ol \alpha_I \in \eval[1]{D_n}_U$.
	Furthermore,
	if $m > 0$ then $\pm \ol \alpha_I \in \eval[1]{D_n}_U$ for all $I \in J$.
	This leads to the classification in the statement.
\end{proof}

Hence there are two cases: if $m > 0$ one has $\pi_1 \bigl( \bm B(\Phi_{\mf h},\Phi_{\mf g}),A \bigr) \simeq \PB^{BC}_d$, while if $m = 0$ then
\begin{equation*}
	\pi_1 \bigl( \bm B(\Phi_{\mf h},\Phi_{\mf g}),A \bigr) \simeq \PB^{BC,D}_{r,s},
\end{equation*}
denoting by $\PB^{BC,D}_{r,s}$ the pure braid group of the hyperplane arrangement of type $(B_r/C_r)D_s$,
i.e.,
the fundamental group of its complement.
Of course if $r = 0$ then this coincides with the pure braid group $\PB^D_s$ of type $D_s$:
in this case the partition of $\ul n$ is discrete,
and so $A$ is actually a regular vector---%
and $\Phi_{\mf h} = \vn$.
Conversely,
if $s = 0$ one gets $\PB^{BC}_r$,
which is the case where no component of the partition is a singleton.

\begin{remark}
	\label{rmk:triality}

	A consistency check when $n = 4$ is in order,
	as there are more symmetries in play (cf.~Rmk.~\ref{rmk:about_D_components}).

	Namely,
	by Thm.~\ref{thm:restricted_arrangement_type_D}~(1.),
	the case $\Delta_{\mf h} = \set{\theta_3,\theta_4}$ yields a crystallographic arrangement of type $BC_2$:
	here $\Phi_{\mf h} \simeq D_2$,
	and $\Phi_{\mf h}^A = \vn$ corresponds to the discrete partition $\ul 2 = \set{1} \cup \set{2}$ (whence $d = 2$).
	But then necessarily,
	e.g.,
	the case $\Delta_{\mf h'} = \set{\theta_1,\theta_3}$ must yield an isomorphic hyperplane arrangement,
	because there is a diagram automorphism matching $\Delta_{\mf h}$ with $\Delta_{\mf h'}$.
	And indeed,
	in the latter case one has $\Phi_{\mf h'} = \Phi_{\mf h'}^A$,
	with associated partition $\ul 4 = \set{1,2} \cup \set{3,4}$,
	so that $r = 2$ and $s = 0$ in Thm.~\ref{thm:restricted_arrangement_type_D}~(2.).
	Analogously,
	the case $\Delta_{\mf h'} = \set{\theta_1,\theta_2,\theta_3}$ yields $\Phi_{\mf h'} = \Phi_{\mf h'}^A = A_3 \sse D_4$,
	which should be crystallographic of type $B_1$,
	as there is an automorphism mapping $\Delta_{\mf h'} \mt \Delta_{\mf h} = \set{\theta_1,\theta_3,\theta_4}$.
	Indeed,
	here the associate partition of $\ul 4$ is indiscrete,
	so that $r = 1$ and $s = 0$ in Thm.~\ref{thm:restricted_arrangement_type_D}~(2.),
	and the hyperplane complement is a copy of the punctured complex line.

	Finally,
	one can check that analogous identifications hold for the other nontrivial orbits of subdiagrams of $\mc D_{D_4}$,
	choosing:
	(i) one peripheral node (of type $A_1$);
	or (ii) a subdiagram of type $A_2$,
	i.e.,
	a connected diagram on two nodes---%
	which necessarily contains the central node.
	The former yields an arrangement of type $(B_1/C_1)D_2$ in $\mb C^3$,\fn{
		The same happens by selecting the central node,
		noting that the partitions $\set{1} \cup \set{2,3} \cup \set{4}$ and $\set{1,2} \cup \set{3} \cup \set{4}$ have the same number of singleton/nonsingleton components,
		cf.~Rmk.~\ref{rmk:noncrystallography}.}~and the latter of type $(B_1/C_1)D_1$ in $\mb C^2$.
	Incidentally,
	beware that the latter is actually crystallographic,
	just in nonstandard form:
	all triples of distinct lines in $\mb C^2$ are simply-transitively permuted by $\PGL_2(\mb C)$ and can be brought into the root-hyperplane arrangement of type $A_2$.
	Explicitly,
	in the notation of~\eqref{eq:exotic_complement} just below,
	one can change coordinates as
	\begin{equation}
		(z,w) \lmt (z_1,z_2,z_3),
		\qquad z_1 \ceqq 3z - w,
		\quad z_2 \ceqq -3z - w,
		\quad z_3 \ceqq 2w,
	\end{equation}
	noting that $z_1 + z_2 + z_3 = 0$,
	and that $z = 0$ if and only if $z_1 = z_2$,
	etc.
	(This transformation however does \emph{not} preserve the orthogonal/unitary structure,
	and indeed the generated reflection groups are different.)
\end{remark}

To study the new example a bit further, consider tuples $\bm z = (z_1,\dc,z_r) \in \mb C^r$ and $\bm w = (w_1,\dc,w_s) \in \mb C^s$.
Then the ``exotic'' hyperplane complement is
\begin{equation}
	\label{eq:exotic_complement}
	X_{r,s} \ceqq \Set{ (\bm z,\bm w) \in \mb C^{r+s} | z_i \neq 0, \quad z_i \neq \pm z_j, z_i \neq \pm w_k, w_k \neq \pm w_l }.\fn{
		The fairly recent~\cite[Thm.~1]{amend_deligne_roehrle_2020_on_the_k_pi_1_problem_for_restrictions_of_complex_reflection_arrangements} proves that these complements are $K(\pi,1)$,
		extending the seminal work~\cite{bessis_2015_finite_complex_reflection_arrangements_are_k_pi_1},
		which in turn establishes the case of---%
		finite---%
		complex reflection arrangements (building on the important past work of many more people,
		cf.~\cite{brieskorn_1973_sur_les_groupes_de_tresses_d_apres_v_i_arnold,
			orlik_terao_1992_arrangements_of_hyperplanes,fadell_neuwirth_1962_configuration_spaces,
			deligne_1972_les_immeubles_des_groupes_de_tresses_generalises,
			nakamura_1983_a_note_on_the_k_pi_1_property_of_the_orbit_space_of_the_unitary_reflection_group_g_m_l_n,
			orlik_solomon_1988_discriminants_in_the_invariant_theory_of_reflection_groups}).}
\end{equation}

Denote then by $\on F_i \ceqq \mb Z^{\ast i}$ the free group on $i \geq 0$ generators.

\begin{proposition}
	\label{thm:exotic_type}

	There is a group isomorphism
	\begin{equation*}
		\PB^{BC,D}_{r,1} \simeq \PB^{BC}_r \ltimes \on F_{2r}.
	\end{equation*}
\end{proposition}

\begin{proof}
	Consider the subspace $X_r \ceqq X_{r,1} \cap \bigl( \mb C^r \times \set{0} \bigr) \sse \mb C^{r+1}$, so that
	\begin{equation*}
		X_r \simeq \Set{ \bm z \in \mb C^r | z_i \neq 0, z_i \neq \pm z_j } \sse \mb C^r,
	\end{equation*}
	which is the root-hyperplane complement of type $B_r$/$C_r$.
	Then there is a canonical projection $p \cl X_{r,1} \to X_r$ with fibres
	\begin{equation*}
		p^{-1}(\bm z) \simeq \Set{ w \in \mb C | w \neq \pm z_i }.
	\end{equation*}
	This yields a locally trivial fibration
	\begin{equation}
		\label{eq:exotic_fibration}
		Y_r \lhra X_{r,1} \lxra{p} X_r, \qquad Y_r \ceqq \mb C \sm \set{ \pm1,\dc,\pm r}.
	\end{equation}

	Now $Y_r$ and $X_r$ are path-connected, and furthermore $X_r$ is a $K(\pi,1)$-space~\cite{deligne_1972_les_immeubles_des_groupes_de_tresses_generalises,brieskorn_saito_1972_artin_gruppen_und_coxeter_gruppen}.
	Hence~\eqref{eq:exotic_fibration} induces an exact sequence of fundamental groups (omitting base points):
	\begin{equation}
		\label{eq:extension_exotic_fundamental_group}
		1 \lra \on F_{2r} \lra \PB^{BC,D}_{r,1} \lxra{\pi_1(p)} \PB^{BC}_r \lra 1.
	\end{equation}

	Finally,
	there is a canonical global (zero) section $X_r \hra X_{r,1}$ splitting~\eqref{eq:extension_exotic_fundamental_group}.
\end{proof}

\begin{remark}[Exceptional isomorphism]
	\label{rmk:exceptional_noncrystallographic_iso}

	If moreover $r = 1$ then~\eqref{eq:extension_exotic_fundamental_group} simplifies to
	\begin{equation*}
		1 \lra \on F_2 \lra \PB^{BC,D}_{1,1} \lra \mb Z \lra 1,
	\end{equation*}
	and in this case we can identify the extension.

	Namely the space $X_{1,1} \sse \mb C^2$ is isomorphic to the root-hyperplane complement of type $A_2$ (cf.~Rmk.~\ref{rmk:triality}).
	Hence $\PB^{BC,D}_{1,1} \simeq \PB_3$, and using $\PB_2 \simeq \mb Z$ one sees that~\eqref{eq:extension_exotic_fundamental_group} becomes the usual split sequence
	\begin{equation*}
		1 \lra \on F_2 \lra \PB_3 \lra \PB_2 \lra 1. \qedhere
	\end{equation*}
\end{remark}

\subsection{Generalised fission trees}

Once more a filtration of fission subsystems splits into $\Phi_{\mf h_1}^A \sse \dm \Phi_{\mf h_p}^A \sse A_{n-1}$ and $D_{m_1} \sse \dm \sse D_{m_p} \sse D_n$, for an increasing sequence of integers $m_i \leq n$.

To encode $\Gamma_Q$ we now need to retain more information, according to the statement of Thm.~\ref{thm:restricted_arrangement_type_D}: namely at each level we must recall the number of trivial/nontrivial type-$A$ irreducible components of $\Phi_{\mf h_l} \cap D_{m_{l+1}} \sse \Phi_{\mf h_{l+1}}$,
in the sense of \S~\ref{sec:type_A_kernels}.
This leads to the following generalisation of Def.~\ref{def:bichromatic_fission_tree},
where we introduce the set $\set{s,l}$ of diameters ``small'' and ``large'', ordered by $s \leq l$:
\begin{definition}[Generalised fission tree]
	\label{def:generalised_tree}

	A \emph{generalised fission tree} is a bichromatic fission tree $\bigl( \mc T_Q,c \bigr)$ equipped with a \emph{diameter function} $d \cl \mc T_0 \to \set{s,l}$; in turn a diameter function satisfies:
	\begin{enumerate}
		\item $d(i) = l$ if $c(i) = b$;

		\item $d \bigl( \bm \phi(i) \bigr) \geq d(i)$ for $i \in \mc T_0 \sm \set{\ast}$;

		\item $k_i \leq 1$ if $d(i) = s$.
	\end{enumerate}
\end{definition}

Hence all blue nodes are large,
while green nodes can be small or large; large green nodes can have (green) child-nodes of any diameter, while small green nodes cannot split.

The algorithm to attach a generalised fission tree to a double filtration as above is the following.
A node $i \in J_l$ corresponds to an irreducible component of the subsystem $\Phi_{\mf h_l} \sse D_n$: put a large green node for each \emph{nontrivial} type-$A$ component, a small green node for each \emph{trivial} type-$A$ component, and finally a large blue node if $m_l > 0$.
The parent-node function is determined as in the bichromatic case.

To compute $\Gamma_Q$ in terms of the tree, note there exists a unique blue node $i_0 \in \mc T_0$ with no blue child-nodes (it might be a leaf): let $r_0, s_0 \geq 0$ be the number of large and small child-nodes of $i_0$, respectively, and set $\mc T_0' \ceqq \mc T_0 \sm \set{i_0}$.
Then retain the notation of \S~\ref{sec:type_BC}.

\begin{theorem}
	\label{thm:tree_gives_wmcg_type_D}
	There is a group isomorphism
	\begin{equation*}
		\Gamma_Q \simeq \prod_{c^{-1}(g)} ( \PB_{k_i} ) \times \PB^{BC,D}_{r_0,s_0} \times \prod_{c^{-1}(b) \cap \mc T_0'} \PB^{BC}_{k_i}.
	\end{equation*}
	Conversely, pure local WMCGs of type $D$ are obtained by adding any one exotic factor to a pure local WMCG of type $B/C$.
\end{theorem}

\begin{proof}
	The new factor, with respect to Thm.~\ref{thm:tree_gives_wmcg_type_BC}, comes from an irreducible type-$D$ component which decomposes into irreducible components of type $A$ only (cf.~Thm.~\ref{thm:restricted_arrangement_type_D}).

	As for the second statement, if the special node $i_0 \in \mc T_0$ is a leaf then $\Gamma_Q$ only depends on the underlying bichromatic fission tree, and yields any pure local WMCG of type $B/C$.
	Then adding a new level where only $i_0$ splits, and has no blue child-nodes, adds an exotic factor~\eqref{eq:exotic_complement} of any kind.
\end{proof}

This is the most general pure local WMCG for a classical simple Lie algebra, containing a factor which is \emph{not} in general the pure braid group of a complex reductive Lie algebra.

\section{Pure cabled braid groups}
\label{sec:braid_operad}

Here we prove the ``multi-scale'' (pure) braiding conjecture in type $A$, cf.~\cite{ramis_2012_iso_irregular_deformations_of_linear_ode_and_dynamics_of_painleve_equations}.
We will do it using the fission trees and the pure braid group operad.

In brief one can express elements of $\Gamma_Q$ as braids on as many strands as the number of leaves of the tree, formalising the driving intuition of the introduction; more precisely, \emph{cabling} will provide an injective group morphism $\Gamma_Q \hra \PB_{\, \abs{J_1}}$, where $J_1 \sse \mc T_0$ are the leaves of $\mc T_Q$.
The final statement is that pure local WMCGs generalise pure \emph{cabled} braid groups (see Def.~\ref{def:pure_cabled_artin_braid_group}), and we still conjecture that they are given by ``braiding of braids'' (see Conj.~\ref{conj:generalised_operads}).

\subsection{Pure cabling}

There are two natural operations on (pure) braids:
\begin{enumerate}
	\item the ``direct sum'', i.e.,
	      the canonical group embedding
	      \begin{equation*}
		      \prod_i \PB_{m_i} \lra \PB_m, \qquad (\sigma_i)_i \lmt \bops_i \sigma_i,
	      \end{equation*}
	      with $m_i \in \mb Z_{\geq 0}$ and $m = \sum_i m_i$;

	\item the ``block braid''
	      \begin{equation*}
		      \PB_m \lra \PB_k, \qquad \sigma \lmt \sigma \braket{ k_1,\dc,k_m },
	      \end{equation*}
	      with $m, k_1, \dc, k_m \in \mb Z_{\geq 0}$ and $k = \sum_i k_i$, which is the function obtained by replacing the $i$-th strand of a braid by $k_i$ parallel copies of it.
\end{enumerate}

Then the \emph{cabling} of a braid $\tau \in PB_m$ onto the $i$-th strand of a braid $\sigma \in PB_n$ is
\begin{equation}
	\label{eq:elementary_cabling}
	\sigma \circ_i \tau \ceqq \sigma \braket{ \underbrace{1,\dc,1}_{i-1 \text{ times}},m,\underbrace{1,\dc,1}_{n-i \text{ times}} \! } \cdot \bigl( \underbrace{\Id_1 \ops \dm \ops \Id_1}_{i-1 \text{ times}} \ops \tau \ops \underbrace{\Id_1 \ops \dm \ops \Id_1}_{n-i \text{ times}} \bigr) \in \PB_{m+n-1},
\end{equation}
where on the rightmost factor $\Id_1 \in \PB_1$.
In words this means replacing the $i$-th strand of $\sigma$ with the braid $\tau$.

One can show the data of the sets $\on{P}(n) = \PB_n$, the unit $\Id_1 \in \PB_1$, and the maps~\eqref{eq:elementary_cabling}, satisfies the associativity/unity axioms of an operad (as introduced in~\cite{boardman_vogt_1968_homotopy_everything_H_spaces,may_1972_the_geometry_of_iterated_loop_spaces,boardman_vogt_1973_homotopy_invariant_algebraic_structures_on_topological_spaces}), leading to the \emph{pure braid group operad} $\ms{P\!B}$~\cite[
	\S~5]{yau_2021_infinity_operads_and_monoidal_categories_with_group_equivariance}.
In particular ``simultaneous'' cabling yields the operadic composition
\begin{equation*}
	\gamma \cl \PB_n \times \prod_{i = 1}^n \PB_{k_i} \lra \PB_m, \quad (\sigma,\tau_1,\dc,\tau_n) \lmt \gamma(\sigma;\tau_1,\dc,\tau_n),
\end{equation*}
where $m = \sum_i k_i$, defined by
\begin{equation}
	\label{eq:artin_braid_operad_composition}
	\gamma(\sigma;\tau_1,\dc,\tau_n) \ceqq \sigma \braket{k_1,\dc,k_n} \cdot (\tau_1 \ops \dm \ops \tau_n).
\end{equation}

In principle this is only a function of sets, but if we equip the domain with the direct-product group structure then:
\begin{lemma}
	\label{lem:injective_morphism_operad_composition}
	The operadic composition~\eqref{eq:artin_braid_operad_composition} is an injective group morphism.
\end{lemma}

\begin{proof}
	The compatibility with products follows from~\cite[Lem.~5.2.4]{yau_2021_infinity_operads_and_monoidal_categories_with_group_equivariance}.\fn{
		This is thus a (noncrossed) ``group'' operad~\cite[Exmp.~2.11]{zhang_group_operads_and_homotopy_theory} (cf.~\cite{yoshida_group_operads_as_crossed_interval_groups,wahl_2001_ribbon_braids_and_related_operads}), a.k.a.~an ``action'' operad~\cite{corner_gurski_operads_with_general_groups_of_equivariance_and_some_2_categorical_aspects_of_operads_in_cat} (cf.~\cite[Def.~4.1.1]{yau_2021_infinity_operads_and_monoidal_categories_with_group_equivariance}).}

	To show injectivity we can prove that if
	\begin{equation*}
		\sigma' = \sigma \braket{ 1,\dc,1,m,1,\dc,1} = \Id_1 \ops \Id_1 \ops \tau \ops \Id_1 \dm \ops \Id_1 \in \PB_{m+n-1},
	\end{equation*}
	for some $(\sigma,\tau) \in \PB_n \times \PB_m$, then both $\sigma$ and $\tau$ are trivial.
	This identity implies the first $i-1$ and the last $n-i$ strands of $\sigma'$ have trivial braiding, so the same is true of all the strands of $\sigma$ except at most the $i$-th one; but if this had nontrivial braiding then the ``central'' $m$ strands of $\sigma'$ would cross the ``peripheral'' ones, and $\sigma' = \sigma \braket{ 1,\dc,1,m,1,\dc,1 }$ is impossible.
\end{proof}

Let now $\mc T$ be a tree with nodes $\mc T_0$, and parent-node function $\bm \phi \cl \mc T_0 \sm \set{\ast} \to \mc T_0$, as in \S~\ref{sec:type_A}.
Retain the notation for the levels $J_l \sse \mc T_0$ and the number $k_i \geq 0$ of child-nodes of $i \in \mc T_0$.

\begin{definition}
	\label{def:pure_cabled_artin_braid_group}
	The \emph{pure cabled braid group} $\ms{P\!B}(\mc T)$ of $\mc T$ is the group obtained at the end of the following sequence of applications of~\eqref{eq:artin_braid_operad_composition}:
	\begin{enumerate}
		\item start at the root and set $\ms{P\!B}(\mc T)_{p+1} \ceqq \PB_1$ (the trivial group);

		\item for each level $l \in \set{p,\dc,1}$ define recursively
		      \begin{equation*}
			      \ms{P\!B}(\mc T)_l \ceqq \gamma \Bigl( \ms{P\!B}(\mc T)_{l+1} \times \prod_{J_{l+1}} \PB_{k_i} \Bigr) \sse \PB_{\, \abs{J_l}}.
		      \end{equation*}
	\end{enumerate}
\end{definition}

By construction $\ms{P\!B}(\mc T) = \ms{P\!B}(\mc T)_1 \sse \PB_{\, \abs{J_1}}$ is a subgroup of the pure braid group on as many strands as the leaves of $\mc T$, and finally matching up fission/cabling trees yields the following.

\begin{theorem}
	\label{thm:wmcg_type_A_equal_cabled_braid_groups}
	In type $A$ there is a group isomorphism $\Gamma_Q \simeq \ms{P\!B}(\mc T_Q)$.
\end{theorem}

\begin{proof}
	By induction on $p \geq 1$, the height of $\mc T_Q$.

	The base uses part of the operad unity axiom, namely the identity
	\begin{equation*}
		\gamma( \PB_1 \times \PB_k ) = \PB_k, \qquad k \in \mb Z_{\geq 1}.
	\end{equation*}

	For the inductive step, the recursive hypothesis yields
	\begin{equation*}
		\ms{P\!B}(\mc T)_2 = \prod_{\mc T_0'} \PB_{k_i} \sse \PB_{\, \abs{J_2}},
	\end{equation*}
	where $\mc T_0' \ceqq \mc T_0 \sm J_1$ are the nodes of the (sub)tree obtained by pruning the leaves; thus
	\begin{equation*}
		\ms{P\!B}(\mc T)_1 = \gamma \Biggl( \prod_{\mc T_0'} \PB_{k_i} \times \prod_{J_2} \PB_{k_i} \Biggr) \simeq \prod_{\mc T_0} \PB_{k_i} \simeq \Gamma_Q,
	\end{equation*}
	using Lem.~\ref{lem:injective_morphism_operad_composition} and Thm.~\ref{thm:tree_gives_wmcg_type_A}.
\end{proof}

\begin{example}[Braiding of Stokes data]
	\label{ex:braiding_stokes_data}
	To showcase future applications, we will write here an explicit formula for an example of braiding of Stokes data, i.e.,
	an action of a pure cabled braid group on a wild character variety.

	Consider again the irregular type $Q$ of~\eqref{eq:example_irregular_type}, for $\mf g = \mf{gl}_3(\mb C)$.
	Recall that
	\begin{equation*}
		\bm B_Q \simeq \Set{ a,a',b,b',c \in \mb C | a \neq a', \, b \neq b' } \simeq \Conf_2(\mb C)^2 \times \mb C,
	\end{equation*}
	and the fission tree $\mc T_Q$ appears in Fig.~\ref{fig:example_fission_tree}.
	Then the pure cabled braid group $\ms{P\!B}(\mc T_Q)$, is generated by the braids of Fig.~\ref{fig:braids_to_be_cabled}.

	Now, using the description of Stokes data by level~\cite[\S~7.2]{boalch_2014_geometry_and_braiding_of_stokes_data_fission_and_wild_character_varieties}, the space of Stokes representations is identified with tuples
	\begin{equation}
		\label{eq:example_stokes_data}
		\rho = \bigl( h,B_1^1,B_3^1,B^2_1,B^2_2,B^2_3,B^2_4 \bigr) \in \GL_3(\mb C)^7, \quad \text{such that} \quad h \cdot (B_3^1B_1^1) \cdot (B^2_4 B^2_3 B^2_2 B^2_1) = 1.
	\end{equation}
	Here $h \in T$ is in the standard maximal torus, and the rest are unipotent elements.

	Then the explicit action of the ``level-2'' generator is
	\begin{equation}
		\label{eq:stokes_braiding_1}
		\sigma \cl \rho \lmt \bigl( h,B_1^1,B_3^1,B_3^2,B_4^2,h_1^{-1}B_1^2h_1,h_1^{-1}B_2^2h_1 \bigr), \qquad h_1 \ceqq hB_3^1B_1^1 \in \GL_3(\mb C),
	\end{equation}
	while the ``level-1'' generator acts by
	\begin{equation}
		\label{eq:stokes_braiding_2}
		\tau_1 \cl \rho \lmt \bigl( h,B_3^1,h^{-1}b_1h,b_1B_1^2b_1^{-1},b_1B_2^2b_1^{-1},b_1B_3^2b_1^{-1},b_1B_4^2b_1^{-1} \bigr), \qquad b_1 \ceqq B_1^1.
	\end{equation}
	One can check that these actions commute,
	and that the (quasi) moment-map condition~\eqref{eq:example_stokes_data} is preserved.

	Moreover,
	the actions~\eqref{eq:stokes_braiding_1}--\eqref{eq:stokes_braiding_2} commute with the diagonal conjugation action of $T \sse G$ on the space of Stokes representations, hence they descend to an action on the quasi-Hamiltonian $T$-quotient, viz.,
	the wild character variety $\mc M_{\on B}$.
	(Analogously,
	the $\PB_m$-action commutes with the diagonal $G$-action on the space of representations of the fundamental group of an $m$-punctured sphere, descending to an action on the tame character variety of \S~\ref{sec:background}.)
\end{example}

We conclude with a precise formulation of the multilevel/nongeneric braiding conjecture, beyond type $A$:

\begin{conjecture}[Classical pure local cabled braid groups]
	\label{conj:generalised_operads}
	There exists a 3-coloured (action) operad $\ms P$, whose evaluations on (a variation of) the generalised fission trees of Def.~\ref{def:generalised_tree} recovers all pure local WMCGs of classical type:
	i.e.,
	there is a group isomorphism $\ms P(\mc T_Q) \simeq \Gamma_Q$, where $Q \in \mf t \ots \ms T_{\Sigma,a}$ is an irregular type of type $A$, $B$, $C$ or $D$.\fn{
		Two colours should be enough to treat type $B/C$, as per Def.~\ref{def:bichromatic_fission_tree}, e.g.,
		using Willwacher's ``moperads''~\cite{willwacher_2016_the_homotopy_braces_formality_morphism} (cf.~\cite{calaque_gonzalez_a_moperadic_approach_to_cyclotomic_associators}),
		seeing the corresponding operads as modules for the type-$A$ pure braid group operad $\ms{P\!\!B}$.
		(Coron's relations with Orlik--Solomon algebras of hyperplane arrangements~\cite{coron_2025_matroids_feynman_categories_and_koszul_duality} might be relevant.)}
\end{conjecture}

\section{Outlook}
\label{sec:outlook}

There is a \emph{full/nonpure} version of local WMCGs, which involves taking out the Weyl action on irregular types, leading to the notion of a ``bare'' irregular type~\cite[Rk.~10.6]{boalch_2014_geometry_and_braiding_of_stokes_data_fission_and_wild_character_varieties} (a.k.a.~an ``irregular class''); and moreover one can define \emph{twisted} irregular types/classes~\cite{boalch_yamakawa_twisted_wild_character_varieties}, leading to ``twisted'' (dressed/bare) wild Riemann surfaces.
A diagram-theoretic description of twisted irregular classes for $G = \GL_n(\mb C)$ was given in~\cite{boalch_yamakawa_2020_diagrams_for_nonabelian_hodge_spaces_on_the_affine_line}, and more generally in~\cite{doucot_diagrams_and_irregular_connections_on_the_riemann_sphere}: their admissible deformations will be considered elsewhere.\fn{
	The general full/nonpure untwisted case is now studied in~\cite{doucot_rembado_2025_topology_of_irregular_isomonodromy_times_on_a_fixed_pointed_curve}, while the type-$A$ twisted case (both pure and full) is the subject of~\cite{boalch_doucot_rembado_2025_twisted_local_wild_mapping_class_groups_configuration_spaces_fission_trees_and_complex_braids},
	recently extended to any structure group $G$ in~\cite{doucot_rembado_yamakawa_twisted_local_g_wild_mapping_class_groups};
	cf.~also~\cite{boalch_2025_counting_the_fission_trees_and_nonabelian_hodge_graphs_untwisted_case} about type-$A$ fission trees.}

Furthermore,
the admissible deformations of wild Riemann surfaces allow for varying the underlying pointed Riemann surface, as in Def.~\ref{def:admissible_deformations}, and we plan to study the topology of the relevant (universal) ``global'' deformation spaces.\fn{
	This has now been taken on in~\cite{doucot_rembado_tamiozzo_moduli_spaces_of_untwisted_wild_riemann_surfaces}.}

\appendix

\section{Some notions/notations we use}
\label{sec:notions_notations}

We collect here some standard material used throughout the paper, also fixing notation.
Besides Bourbaki, cf.~\cite[\S~2]{collingwood_mcgovern_1993_nilpotent_orbits_in_semisimple_lie_algebras} and~\cite{humphreys_1972_introduction_to_lie_algebras_and_representation_theory}.

\subsection*{About Lie algebras}

Let $\mf g$ be a finite-dimensional reductive Lie algebra.

The \emph{centraliser} of a subset $S \sse \mf g$ is
\begin{equation*}
	\mf Z_{\mf g}(S) = \Set{ X \in \mf g | [X,S] = 0 } \sse \mf g,
\end{equation*}
and in particular $\mf Z_{\mf g} = \mf Z_{\mf g}(\mf g)$ is the centre.
There is a Lie algebra decomposition $\mf g = \mf g' \ops \mf Z_{\mf g}$, where $\mf g' = [\mf g,\mf g]$ is the \emph{semisimple part} of $\mf g$.

A \emph{Cartan subalgebra} $\mf t \sse \mf g$ is a maximal abelian subalgebra consisting of semisimple elements, so $\mf Z_{\mf g} \sse \mf t$.
The Cartan subalgebra decomposes as $\mf t = \mf t' \ops \mf Z_{\mf g}$, where $\mf t' = \mf t \cap \mf g'$, which is a Cartan subalgebra of $\mf g'$.
Conversely the Cartan subalgebras \emph{split} $\mf g$ and $\mf g'$, and the pairs $(\mf g,\mf t)$ and $(\mf g',\mf t')$ are \emph{split} Lie algebras.

The \emph{rank} of $\mf g$ is the dimension of a Cartan subalgebra, so in particular
\begin{equation*}
	\rk(\mf g) = \dim_{\mb C}(\mf Z_{\mf g}) + \rk(\mf g'),
\end{equation*}
while $\rk(\mf g')$ is the \emph{semisimple rank} of $\mf g$.

Given a (reductive) split Lie algebra $(\mf g,\mf t)$, a \emph{root} is a linear functional $\alpha \in \mf t^{\dual} \sm \set{0}$ such that the following subspace is nonzero:
\begin{equation*}
	\mf g_{\alpha} \ceqq \Set{ X \in \mf g | [A,X] = \Braket{ \alpha | A } \cdot X \text{ for } A \in \mf t } \sse \mf g.
\end{equation*}
The (finite) set of roots is denoted by $\Phi_{\mf g} = \Phi(\mf g,\mf t) \sse \mf t^{\dual}$, and is the \emph{root system} of $(\mf g,\mf t)$.
Note that all roots vanish on the centre, so the root system does \emph{not} span $\mf t^{\dual}$ in the nonsemisimple case.
Conversely, if $\Phi_{\mf g'} \sse (\mf t')^{\dual}$ is the (spanning) root system of the semisimple part $\mf g' \sse \mf g$, then its elements are precisely the restriction (to $\mf t'$) of the elements of $\Phi_{\mf g}$.

\subsection*{About root systems}

Hence we will consider root systems $\Phi \sse V$, where $V$ is a finite-dimensional vector space, which are not necessarily irreducible or spanning.
The subspace $\spann_{\mb C}(\Phi) \sse V$ is the \emph{essential} part of $(V,\Phi)$,\fn{
	On this space the root-hyperplane arrangement is ``essential''~\cite[\S~2.B]{broue_malle_rouquier_1998_complex_reflection_groups_braid_groups_hecke_algebras}.}
and the \emph{rank} of $\Phi$ is the dimension of the essential part.
(This coincides with the \emph{semisimple} rank of the associated reductive Lie algebra.)

A \emph{root subsystem} $\Phi' \sse \Phi$ is a subset which is stable under the reflections associated with the roots that it contains.
A root subsystem $\Phi' \sse \Phi$ is \emph{Levi} if it closed under $\mb Q$-linear combinations of its elements, provided these are still roots: in our setting we typically view them as the annihilators of elements $X \in \mf t$---intersected with $\Phi$.
The root/Levi subsystems are permuted by the group $\Aut(\Phi) \sse \GL_{\mb C}(V)$ of automorphisms of $\Phi$, so in particular by the Weyl group, cf.~\cite{oshima_a_classification_of_subsystems_of_a_root_system}.

The \emph{direct sum} of two root systems $(\Phi_1,V_1), (\Phi_2,V_2)$ is
\begin{equation}
	\label{eq:direct_sum_root_systems}
	\Phi_1 \ops \Phi_2 = (V_1,\Phi_1) \ops (V_2,\Phi_2) \ceqq (V_1 \ops V_2,\Phi_1 \mathsmaller{\coprod} \Phi_2).
\end{equation}

We will also encounter \emph{nonreduced} root systems.
There exists a unique (spanning) irreducible nonreduced rank-$n$ root system, up to isomorphism, denoted by $BC_n$: it consists of the vectors
\begin{equation*}
	\Set{ \pm(e_i - e_j), \, \pm (e_i + e_j) | 1 \leq i < j \leq n } \cup \Set{ \pm e_i, \, \pm 2e_i | 1 \leq i \leq n } \sse V = \mb C^n,
\end{equation*}
using the canonical basis of $\mb C^n = \bops_i \mb Ce_i$~\cite[Ch.~VI, \S~4.14]{bourbaki_1968_groupes_et_algebres_de_lie_chapitres_iv_v_vi}.

\subsection*{About braid groups and hyperplane arrangements/complements}

We denote by $\PB_n = \PB^A_n$ the \emph{pure} (Artin) \emph{braid group} on $n \geq 0$ strands~\cite{artin_1925_theorie_der_zoepfe}. (So $\PB_0$ and $\PB_1$ are trivial.)
It is the fundamental group of the space
\begin{equation}
	\label{eq:type_A_complement}
	\Conf_n(\mb C) = \mb C^n \, \mathbin{\big\backslash} \, \bigcup_{1 \leq i \neq j \leq n} H_{ij}, \qquad H_{ij} \ceqq \Set{ (z_1,\dc,z_n) \in \mb C^n \mid z_i = z_j } \sse \mb C^n,
\end{equation}
i.e.,
the space of configurations of (ordered) $n$-tuples points in the complex plane~\cite{fadell_neuwirth_1962_configuration_spaces}.
These are thus the fundamentals group of complements of hyperplane arrangements, i.e.,
``hyperplane complements'' for short.

More generally,
for a split reductive Lie algebra $(\mf g,\mf t)$ we consider the \emph{pure} $\mf g$-\emph{braid group} $\PB_{\mf g}$, which is the fundamental group of the space
\begin{equation}
	\label{eq:regular_cartan}
	\mf t_{\reg} = \mf t \, \mathbin{\big\backslash} \,  \bigcup_{\Phi_{\mf g}} \Ker(\alpha) \sse \mf t,
\end{equation}
viz.,
the complement of the root-hyperplane arrangement---the ``root-hyperplane complement''~\cite{brieskorn_1971_die_fundamentalgruppe_des_raumes_der_regulaeren_orbits_einer_endlichen_komplexen_spiegelungsgruppe,brieskorn_1973_sur_les_groupes_de_tresses_d_apres_v_i_arnold,deligne_1972_les_immeubles_des_groupes_de_tresses_generalises}.
Such arrangements are said to be \emph{crystallographic}, and in particular~\eqref{eq:type_A_complement} corresponds to a simple Lie algebra of type $A_{n-1}$.

In the case of simple Lie algebras of classical type $B_n/C_n$ (resp.,
$D_n$),
we denote by $\PB^{BC}_n$ the pure $\mf g$-braid group (resp.,
$\PB^D_n$).
(Note that types $B_n$ and $C_n$ yield the same arrangement/complement~\eqref{eq:regular_cartan}.)

The Weyl groups are the (rational) reflection groups generated by the root-hyperplane arrangements.
Conversely a reflection group is \emph{crystallographic} if it is the Weyl group of a root system~\cite[Ch.~VI, \S~2.5]{bourbaki_1968_groupes_et_algebres_de_lie_chapitres_iv_v_vi}.

\section{Some remarks about quantisation}
\label{sec:quantisation}

For completeness we review here some of the literature about the quantum analogue of the background material.
While we do not need/use in this paper, it inspires part of this work.

\vspace{5pt}

The main idea is that the nonlinear monodromy actions of mapping class/braid groups on wild character varieties have linear analogues obtained after \emph{quantisation}.
In turn, this means considering the Poisson varieties as the phase-spaces of classical mechanical systems (parameterising pure states), or as spaces of classical gauge fields, and replace them by their analogue in quantum mechanics/field theory (cf.~\cite[App.~A]{rembado_2018_quantisation_of_moduli_spaces_and_connections} and~\cite{faddeev_yakubovski_2009_lectures_on_quantum_mechanics_for_mathematics_students} for the basic mathematical dictionary).

Several constructions are possible, crossing the boundary between mathematics and theoretical physics: rigorous mathematical approaches include \emph{geometric} quantisation, born out of the work of Kirillov, Kostant and Souriau on coadjoint orbits~\cite{kirillov_1962_unitary_representations_of_nilpotent_lie_groups,kostant_1970_quantization_and_unitary_representations_i_prequantization,souriau_1970_structure_des_systemes_dynamiques}, and \emph{deformation} quantisation~\cite{bayen_flato_fronsdal_lichnerowicz_sternheimer_1978_deformation_theory_and_quantization_i_deformations_of_symplectic_structures,bayen_flato_fronsdal_lichnerowicz_sternheimer_1978_deformation_theory_and_quantization_ii_physical_applications}---which concentrates on the quantisation of observables.
In any formalism, quantisation replaces a Poisson/symplectic fibre bundle by a family of complex vector spaces: the mathematician's aim is to prove these assemble into a vector bundle, and equip it with a flat (projective) connection, whose monodromy finally provides (projective) ``quantum'' representations of the fundamental group of the base.

In this framework we notably find the Knizhnik--Zamolodchikov connection (KZ~\cite{knizhnik_zamolodchikov_1984_current_algebra_and_wess_zumino_model_in_two_dimensions,etingof_frenkel_kirillov_1998_lectures_on_representation_theory_and_knizhnik_zamolodchikov_equations}), in the genus-zero Wess--Zumino--Novikov--Witten model (WZNW~\cite{wess_zumino_1971_consequences_of_anomalous_ward_identities,novikov_1982_the_hamiltonian_formalism_and_a_multivalued_analogue_of_morse_theory,witten_1983_global_aspects_of_current_algebra,witten_1984_nonabelian_bosonization_in_two_dimensions}) for 2d conformal field theory~\cite{belavin_polyakov_zamolodchikov_1984_infinite_conformal_symmetry_in_two_dimensional_quantum_field_theory,segal_1988_the_definition_of_conformal_field_theory}, quantising the Schlesinger system~\cite{reshetikhin_1992_the_knizhnik_zamolodchikov_system_as_a_deformation_of_the_isomonodromy_problem,harnad_1996_quantum_isomonodromic_deformations_and_the_knizhnik_zamolodchikov_equations}.
But also the connection of Felder--Markov--Tarasov--Varchenko (FMTV~\cite{felder_markov_tarasov_varchenko_2000_differential_equations_compatible_with_kz_equations}), quantising the system of JMMS~\cite{rembado_2019_simply_laced_quantum_connections_generalising_kz}.

Note that the dual version of the Schlesinger system, which is a particular case of JMMS, was quantised earlier~\cite{boalch_2002_g_bundles_isomonodromy_and_quantum_weyl_groups}, recovering the ``Casimir'' connection of De Concini/Millson--Toledano Laredo (DMT~\cite{millson_toledanolaredo_2005_casimir_operators_and_monodromy_representations_of_generalised_braid_groups});
cf.~the introduction of~\cite{rembado_2019_simply_laced_quantum_connections_generalising_kz}.

The monodromy of KZ then features in the Kohno--Drinfel'd theorem~\cite{kohno_1987_monodromy_representations_of_braid_groups_and_yang_baxter_equations,drinfeld_1989_quasi_hopf_algebras}.
Analogously, the monodromy of DMT is tantamount to a Kohno--Drinfel'd theorem for $q$-Weyl groups~\cite{toledanolaredo_2002_a_kohno_drinfeld_theorem_for_quantum_weyl_groups}, and recovers the action of Lusztig/Soibelman/Kirillov--Reshetikhin~\cite{lusztig_1990_quantum_groups_at_roots_of_1,soibelman_1990_algebra_of_functions_on_a_compact_quantum_group_and_its_representations,kirillov_reshetikhin_1990_q_weyl_group_and_a_multiplicative_formula_for_universal_r_matrices} on the Jimbo--Drinfel'd quantum group $U_q \mf g$~\cite{jimbo_1985_a_q_difference_analogue_of_u_g_and_the_yang_baxter_equation,drinfeld_1987_quantum_groups}.
This Hopf algebras quantises the algebra of functions on $G^*$: as mentioned above, this is the corresponding wild character variety, and the semiclassical limit of the action coincides with De Concini--Kac--Procesi's~\cite{deconcini_kac_procesi_1992_quantum_coadjoint_action}.
(The monodromy of FMTV was considered in~\cite{xu_2020_stokes_phenomenon_and_yang_baxter_equations}.)

These examples are in the regular singular case, and in the generic irregular singular case (cf.~\cite{felder_rembado_2023_singular_modules_for_affine_lie_algebras_and_applications_to_irregular_wznw_conformal_blocks} for arbitrary polar divisor):\fn{
	The high-genus nonsingular case also fits this story: indeed one finds the connection of Witten~\cite{witten_1991_quantization_of_chern_simons_gauge_theory_with_complex_gauge_group,andersen_gammelgaard_the_hitchin_witten_connection_and_complex_quantum_chern_simons_theory} in \emph{complex} quantum Chern--Simons gauge theory, which is equivalent to a ``complexified'' Hitchin connection in genus one~\cite{andersen_malusa_rembado_2022_genus_one_complex_quantum_chern_simons_theory} (cf.~\cite{andersen_malusa_rembado_2024_sp_1_symmetric_hyperkaehler_quantisation}).
	The original Hitchin connection~\cite{hitchin_1990_flat_connections_and_geometric_quantization,axelrod_dellapietra_witten_1991_geometric_quantization_of_chern_simons_gauge_theory} is for \emph{compact} Chern--Simons, so it requires starting from a (maximal) compact subgroup $K \sse G$; nonetheless it yields ``quantum'' representations of mapping class groups, cf.~\cite{masbaum_2003_quantum_representations_of_mapping_class_groups,andersen_2006_asymptotic_faithfulness_of_the_quantum_su_n_representations_of_the_mapping_class_groups,marche_introduction_to_quantum_representations_of_mapping_class_groups}.

	Note that an identification of the projectively-flat vector bundles of nonabelian $\theta$-functions and WZNW conformal blocks was first given in~\cite{beauville_laszlo_1994_conformal_blocks_and_generalized_theta_functions,faltings_1994_a_proof_for_the_verlinde_formula,laszlo_1998_hitchin_s_and_wzw_connections_are_the_same}, in absence of marked points---thus missing KZ.
	Recently this has been extended to the case of marked points, see~\cite{egsgaard_2015_hitchin_connections_for_genus_0_quantum_representations,biswas_mukhopadhyay_wentworth_2023_a_hitchin_connection_on_nonabelian_theta_functions_for_parabolic_g_bundles} and references therein.}
the main result of~\cite{rembado_2019_simply_laced_quantum_connections_generalising_kz,rembado_2020_symmetries_of_the_simply_laced_quantum_connections_and_quantisation_of_quiver_varieties,yamakawa_2022_quantization_of_simply_laced_isomonodromy_systems_by_the_quantum_spectral_curve_method} is that one can also do quantisation in the nongeneric case, namely quantising the system of~\cite{boalch_2012_simply_laced_isomonodromy_systems}, generalising all the above.
(Cf.~also~\cite{calaque_felder_rembado_wentworth_wild_orbits_and_generalised_singularity_modules_stratifications_and_quantisation,chaffe_rembado_yamakawa_wild_genus_zero_quantum_de_rham_spaces} for an extension to arbitrary polar divisors/structure groups.)

In all these important cases one sees that the local wild moduli, viz.,
the irregular types, behave as the moduli of the underlying pointed surface even after quantisation.
But the ``semiclassical'' theory of isomonodromic deformations goes beyond these examples, and it is still expressed in a language that lends itself to quantisation.
This provides a guide to prove analogous statements in a (much) more general context.

\bibliographystyle{amsplain}
\bibliography{/home/gabriele/Desktop/bibliography_macros/bibliography}
\end{document}